\documentclass[11pt,twoside]{article}

\usepackage{amsfonts,amsmath,amssymb,amsthm}
\usepackage{fullpage}
\usepackage{graphics,graphicx}
\usepackage{cases}
\usepackage{comment}

\newlength{\widebarargwidth}
\newlength{\widebarargheight}
\newlength{\widebarargdepth}
\DeclareRobustCommand{\widebar}[1]{%
  \settowidth{\widebarargwidth}{\ensuremath{#1}}%
  \settoheight{\widebarargheight}{\ensuremath{#1}}%
  \settodepth{\widebarargdepth}{\ensuremath{#1}}%
  \addtolength{\widebarargwidth}{-0.3\widebarargheight}%
  \addtolength{\widebarargwidth}{-0.3\widebarargdepth}%
  \makebox[0pt][l]{\hspace{0.3\widebarargheight}%
    \hspace{0.3\widebarargdepth}%
    \addtolength{\widebarargheight}{0.3ex}%
    \rule[\widebarargheight]{0.95\widebarargwidth}{0.1ex}}%
  {#1}}

\usepackage[colorlinks]{hyperref}

\newcommand{\condind}{\ensuremath{\perp\!\!\!\perp}}
\newcommand{\opnorm}[1]{\left|\!\left|\!\left|{#1}\right|\!\right|\!\right|}
\newcommand{\Cov}{\ensuremath{\operatorname{Cov}}}
\newcommand{\Var}{\ensuremath{\operatorname{Var}}}
\newcommand{\E}{\ensuremath{\mathbb{E}}}
\newcommand{\mprob}{\ensuremath{\mathbb{P}}}

\newcommand{\betahat}{\ensuremath{\widehat{\beta}}}
\newcommand{\betastar}{\ensuremath{\beta^*}}
\newcommand{\betatil}{\ensuremath{\widetilde{\beta}}}
\newcommand{\nuhat}{\ensuremath{\widehat{\nu}}}
\newcommand{\real}{\ensuremath{\mathbb{R}}}
\newcommand{\defn}{\ensuremath{:=}}
\newcommand{\inprod}[2]{\ensuremath{\langle #1 , \, #2 \rangle}}
\newcommand{\sign}{\ensuremath{\operatorname{sign}}}
\newcommand{\supp}{\ensuremath{\operatorname{supp}}}

\newcommand{\order}{{\mathcal{O}}}

\newcommand{\nutil}{\ensuremath{\widetilde{\nu}}}
\newcommand{\Loss}{\ensuremath{\mathcal{L}}}

\newcommand{\Ball}{\ensuremath{\mathbb{B}}}
\newcommand{\var}{\ensuremath{\operatorname{var}}}

\newcommand{\ftil}{\ensuremath{\widetilde{f}}}
\newcommand{\scriptD}{\ensuremath{\mathcal{D}}}

\newcommand{\htil}{\ensuremath{\widetilde{h}}}

\newcommand{\sigmahat}{\ensuremath{\widehat{\sigma}}}
\newcommand{\sigmastar}{\ensuremath{\sigma^*}}

\newcommand{\Sigmahat}{\ensuremath{\widehat{\Sigma}}}

\newcommand{\scriptT}{\ensuremath{\mathcal{T}}}
\newcommand{\ellbar}{\ensuremath{\widebar{\ell}}}
\newcommand{\sigmatil}{\ensuremath{\widetilde{\sigma}}}
\newcommand{\sigmabar}{\ensuremath{\widebar{\sigma}}}

\newcommand{\scriptJ}{\ensuremath{\mathcal{J}}}

\newcommand{\muhat}{\ensuremath{\widehat{\mu}}}
\newcommand{\bhat}{\ensuremath{\widehat{b}}}
\newcommand{\Thetahat}{\ensuremath{\widehat{\Theta}}}
\newcommand{\scriptN}{\ensuremath{\mathcal{N}}}
\newcommand{\mvec}{\ensuremath{\operatorname{vec}}}

\newcommand{\scriptC}{\ensuremath{\mathcal{C}}}

\newcommand{\Ahat}{\ensuremath{\widehat{A}}}
\newcommand{\Btil}{\ensuremath{\widetilde{B}}}
\newcommand{\scriptS}{\ensuremath{\mathcal{S}}}
\newcommand{\tr}{\ensuremath{\operatorname{tr}}}
\newcommand{\Gammastar}{\ensuremath{\Gamma^*}}
\newcommand{\Sigmastar}{\ensuremath{\Sigma^*}}
\newcommand{\Thetastar}{\ensuremath{\Theta^*}}

\newcommand{\uhat}{\ensuremath{\widehat{u}}}
\newcommand{\scriptB}{\ensuremath{\mathcal{B}}}
\newcommand{\Lep}{\ensuremath{\operatorname{Lep}}}
\newcommand{\med}{\ensuremath{\operatorname{med}}}
\newcommand{\MAD}{\ensuremath{\operatorname{MAD}}}
\newcommand{\deltahat}{\ensuremath{\widehat{\delta}}}

\newtheorem{lem*}{Lemma}
\newtheorem{thm*}{Theorem}
\newtheorem{alg*}{Algorithm}
\newtheorem{cor*}{Corollary}
\newtheorem{prop*}{Proposition}
\newtheorem{rem}{Remark}
\newtheorem{assumption}{Assumption}
\newtheorem{definition}{Definition}
\newtheorem{example*}{Example}

\begin{document}

\begin{center}
	
{\bf{\LARGE{Scale calibration for high-dimensional robust regression}}}
	
\vspace*{.2in}

\begin{tabular}{c}
{\large{Po-Ling Loh}} \\
{\large{\texttt{loh@ece.wisc.edu}}} \\
\vspace*{.005in} \\
Departments of Electrical \& Computer Engineering and Statistics \\
University of Wisconsin - Madison \\
Madison, WI 53706
\end{tabular}

\vspace*{.2in}

November 5, 2018

\end{center}

\begin{abstract}

We present a new method for high-dimensional linear regression when a scale parameter of the additive errors is unknown. The proposed estimator is based on a penalized Huber $M$-estimator, for which theoretical results on estimation error have recently been proposed in high-dimensional statistics literature. However, the variance of the error term in the linear model is intricately connected to the optimal parameter used to define the shape of the Huber loss. Our main idea is to use an adaptive technique, based on Lepski's method, to overcome the difficulties in solving a joint nonconvex optimization problem with respect to the location and scale parameters.

\end{abstract}


\section{Introduction}

Robust statistics, in its classical form, is a mature and established field~\cite{Hub81, MarEtal06, HamEtal11}. Recently, notions from robust statistics such as $\epsilon$-contamination and influence functions have surfaced in theoretical computer science and machine learning~\cite{EliPon03, KohLia17}. The use of the Huber loss in place of a squared error loss to encourage robustness has long been adopted in engineering fields, as well~\cite{Fes06}.

In statistics, a small but growing body of work concerns analyzing high-dimensional analogs of classical robust estimators \cite{KhaEtal07, WanEtal07, Mar11, Bra15, Loh17, FanEtal17, SmuYoh17, SunEtal17, FreEtal17}. The basic premise is that although it is relatively straightforward to devise reasonable high-dimensional estimators, theoretical analysis may become somewhat trickier in high dimensions \cite{MedMar14}. Furthermore, special care must be taken when optimizing such objective functions over a high-dimensional space \cite{AlfEtal13}.

Our previous work \cite{Loh17} developed a theory for robust high-dimensional linear regression estimators using penalized $M$-estimation. The main contribution was to show that global optima of $\ell_1$-penalized $M$-estimators enjoy the same rates of convergence as minimizers of the Lasso program, when the $M$-estimation loss function is convex and has a bounded derivative---without requiring a Gaussian or sub-Gaussian assumption on the additive errors. In fact, we also established that local optima of penalized $M$-estimators with a nonconvex, bounded-derivative loss are statistically consistent within a constant-radius region of the global optimum, and such local optima may be obtained via a two-step process initialized using a global optimum of the $\ell_1$-penalized Huber loss.

However, a drawback of Loh~\cite{Loh17}, as well as other related work on penalized $M$-estimation \cite{FanEtal17, SunEtal17}, is that the theoretically optimal choice of the parameter involved in defining the Huber loss depends crucially on the scale of the additive errors. This should not be surprising, given that similar complications were recognized in low-dimensional settings for location estimation when prior knowledge of the scale was unavailable~\cite{Hub64}. The ``adaptive" methods proposed for low-dimensional robust regression \cite{Jae71, Hog74} are mostly heuristic suggestions involving, for example, computing the Huber regression estimate over a grid of values and choosing the parameter that minimizes a surrogate for asymptotic variance.
Even in low dimensions, a theoretical gap has remained in terms of how to rigorously calibrate the Huber loss function in a finite-sample setting.

In this paper, we introduce a new solution to the problem of adaptively choosing the scale parameter of a robust $M$-estimator. The key tool is Lepski's method, and the key observation is that whenever the Huber loss parameter is larger than the true scale parameter of the additive errors, it is possible to derive $\ell_1$- and $\ell_2$-error bounds on the global optimum that increase linearly with the choice of Huber parameter. This allows us to apply Lepski's method to obtain an estimator that behaves at least as well as an oracle estimator. Importantly, our method bypasses the hard optimization problem of jointly estimating the location and scale. We note that Lepski's method could also be invoked in the low-dimensional, unpenalized setting to rigorously obtain robust regression estimators without needing to optimize a nonconvex problem in an ad hoc manner.

\textbf{Related work:} Other proposals for regression with heavy-tailed errors includes work by Hsu and Sabato~\cite{HsuSab16}, Minsker~\cite{Min15}, and Lugosi and Mendelson~\cite{LugMen17}. However, many of these methods focus on situations where the covariates are well-behaved, and all of them assume that an upper bound on the error variance is known. In contrast, our method produces consistent estimators under extremely weak assumptions on the covariates, and encompasses situations where preliminary scale estimates are notoriously difficult to obtain.

Another important related work is by Chichignoud et al.~\cite{ChiEtal16}, who suggest an adaptive method for tuning parameter selection in the Lasso based on Lepski's method. However, the main focus in that paper is in obtaining near-optimal bounds on the $\ell_\infty$-error. Importantly, the objective function still involves a least-squares loss as in the classical Lasso, whereas our objective functions are designed for robust regression and have the corresponding parameter linked to the regularization parameter involved in the $\ell_1$-norm.

\textbf{Notation:} We write $a \precsim b$ if $a$ is less than or equal to $b$, up to a positive constant, and we define $a \succsim b$ analogously. For a vector $v \in \real^p$, we write $\supp(v) \subseteq \{1, \dots, p\}$ to denote the support of $v$, and for an arbitrary subset $S
\subseteq \{1, \dots, p\}$, we write $v_S \in \real^S$ to denote the vector $v$ restricted to $S$. For a matrix $M$, we write $\opnorm{M}_{q}$ to denote the $\ell_q$-operator norm, and we write $\|M\|_{\max}$ to denote the elementwise $\ell_\infty$-norm. We write $\mvec(M)$ to denote the vectorized version of the matrix. Finally, we use the notation $c, C', c_0$, etc., to denote universal positive constants, where we may use the same notation to refer to different constants as we move between results. We use the abbreviation ``w.h.p." to refer to an event occurring with probability tending to 1.


\section{Background and problem setup}

We begin by describing the regression model to be studied in our paper. We also discuss several previously existing proposals in the literature.

\subsection{Model and assumptions}

Consider observations $\{(x_i, y_i)\}_{i=1}^n$ from the linear model
\begin{equation}
\label{EqnLinear}
y_i = x_i^T \betastar + \epsilon_i,
\end{equation}
where $\betastar \in \real^p$ is the unknown regression parameter vector and the $x_i$'s and $\epsilon_i$'s are i.i.d.\ draws from covariate and error distributions, such that $x_i \condind \epsilon_i$ and $\E[\epsilon_i] = \E[x_i] = 0$. In other words, we assume a ``random carrier" rather than fixed design model of regression. Note that our results could be adapted to the fixed design setting in a fairly straightforward manner; however, we are primarily interested in a setting where the distribution of the covariates comes from a heavy-tailed distribution, leading to high-leverage points. We assume that $\|\betastar\|_0 \le k$, where $k < n \ll p$, and denote $S := \supp(\betastar)$.

We will assume that the distribution of $\epsilon_i$ is symmetric, which will be required---as in the case of classical robust regression analysis---for our $M$-estimators to be consistent. We will introduce additional assumptions on the distributions of the $\epsilon_i$'s and $x_i$'s in Section~\ref{SecProbCond}.


\subsection{Previous work}

We now briefly describe several previously proposed methods for robust linear regression in high dimensions. We focus on methods that have been devised to handle outliers in the covariates, since our proposed algorithm is provably consistent in such settings, as well. (For additional related work, see the references cited in the introduction.)

The sparse least trimmed squares (LTS) estimator \cite{AlfEtal13} aims to optimize the objective
\begin{equation*}
\betahat \in \arg\min_{\beta}\left\{\frac{1}{h} \sum_{i=1}^h r_{(i)}^2 + \lambda \|\beta\|_1\right\},
\end{equation*}
where the $r_{(i)}^2$'s are the sorted residuals $\{(y_i - x_i^T \beta)^2\}$ in ascending order, and $h \le n$ is a truncation parameter. This is an $\ell_1$-penalized version of the least trimmed squares estimator \cite{Rou84}. Although sparse LTS has been shown to perform well in simulations, only a heuristic algorithm has been proposed for optimizing the objective, and statistical properties of both global and local optima is absent from the literature.

The $S$-Bridge estimator \cite{Mar11, SmuYoh17} is defined via the objective function
\begin{equation*}
\betahat \in \arg\min_\beta \left\{s^2(r(\beta)) + \lambda \|\beta\|_r^r\right\},
\end{equation*}
where $r > 0$, and $s(r(\beta))$ is a robust scale estimator based on the residuals $\{y_i - x_i^T \beta\}$. The $MM$-Bridge estimator is defined by
\begin{equation*}
\betahat \in \arg\min_\beta \left\{\frac{1}{n} \sum_{i=1}^n \rho_1\left(\frac{r_i(\beta)}{s(r(\betahat_1))}\right) + \lambda \|\beta\|_r^r\right\},
\end{equation*}
where $\rho_1$ is a robust loss function and $\betahat_1$ is an initial estimate of $\betastar$; for $r = 1$, this method is also known as the $MM$-Lasso. Smucler and Yohai~\cite{SmuYoh17} derived the asymptotic consistency of global optima when the loss function $\rho_1$ is of a redescending type, meaning that $\rho_1'$ is eventually equal to 0. However, the results are asymptotic, and again, no guarantees are provided for the performance of local optima, which may result from the optimization algorithm proposed by the authors. Penalized $S$-estimators are further analyzed in Freue et al.~\cite{FreEtal17}.

Our work builds upon Loh~\cite{Loh17}, which studied local and global optima of penalized $M$-estimators. The main contribution in that work is a rigorous nonasymptotic analysis of global optima in the convex case, as well as an analysis of certain consistent local optima when the objective function is nonconvex. However, the success of the methods proposed in that paper require the parameter of the Huber loss to be chosen correctly, i.e., upper-bounding an expression involving moments and tails of the error distribution. Since this information would generally be unknown a priori, the question of how to choose the Huber parameter in an adaptive manner remained unanswered.

Finally, we mention methods based on joint estimation of location and scale. One natural approach is to jointly minimize the objective function
\begin{equation*}
(\betahat, \sigmahat) \in \arg\min_{\beta, \sigma} \left\{\frac{1}{n} \sum_{i=1}^n \ell\left(\frac{y_i - x_i^T \beta}{\sigma}\right)\right\}
\end{equation*}
(or a high-dimensional analog thereof). However, even when the loss function is convex, this leads to a highly nonconvex objective. Iteratively optimizing with respect to $\beta$ and $\sigma$ motivates the $MM$-estimator~\cite{Yoh87}, but theoretical guarantees in terms of both statistical consistency and convergence of the optimization algorithm are largely absent from the literature. Huber~\cite{Hub81} also proposed the concomitant estimator:
\begin{equation}
\label{EqnHubConcom}
(\betahat, \sigmahat) \in \arg\min_{\beta, \sigma} \left\{\frac{1}{n} \sum_{i=1}^n \ell\left(\frac{y_i - x_i^T \beta}{\sigma}\right) \sigma + a\sigma \right\},
\end{equation}
where $a$ is an appropriate constant. The key insight is that if $\ell$ is a convex function, the loss function $\Loss_n(\beta, \sigma)$ defined here is also jointly convex in $(\beta, \sigma)$. 
%
%
However, the choice of the correct constant $a$ to provide consistency is a bit more tricky. A small calculation shows that if we denote $\Loss(\beta, \sigma) = \E\left[\Loss_n(\beta, \sigma)\right]$, we have $\nabla \Loss(\betastar, \sigmastar) = 0$ provided
%
%
\begin{equation*}
a = \E\left[\ell'\left(\frac{\epsilon_i}{\sigmastar}\right) \frac{\epsilon_i}{\sigmastar} - \ell\left(\frac{\epsilon_i}{\sigmastar}\right)\right]
\end{equation*}
holds. Thus, some prior knowledge of the distribution of $\epsilon_i$ is required to choose $a$ appropriately. In contrast, our method results in a consistent estimate of $\betastar$ whenever $\epsilon_i$ has a symmetric distribution.
%
Another important issue is that if $\ell$ is nonconvex---as is recommended to deal with high-leverage points in the covariates---Huber's estimator~\eqref{EqnHubConcom} would no longer be jointly convex, leading to a more tricky analysis of local optima in the $(\beta, \sigma)$ parameter space.

\section{Adaptive scale estimation}

Consider the Huber loss function
\begin{equation*}
\ell_\tau(u) = \begin{cases}
\frac{u^2}{2}, & \text{if } |u| \le \tau, \\
\tau |u| - \frac{\tau^2}{2}, & \text{if } |u| > \tau,
\end{cases}
\end{equation*}
defined with respect to a parameter $\tau > 0$. Importantly, the Huber loss is differentiable, and
$\|\ell'_\tau\|_\infty \le \tau$. We also define the weight function $w: \real^p \rightarrow \real$, such that
\begin{equation*}
w(x) = \min\left\{1, \frac{b}{\|Bx\|_2}\right\},
\end{equation*}
where $b \in \real$ and $B \in \real^{p \times p}$ are fixed constants. Importantly,
\begin{equation*}
\|w(x) x\|_2 \le \frac{b\|x\|_2}{\|Bx\|_2} \le \frac{b}{\lambda_{\min}(B)} := b'.
\end{equation*}
We form the $\ell_1$-regularized Huber estimator
\begin{equation}
\label{EqnHuberReg}
\betahat_\tau \in \arg\min_\beta \left\{\frac{1}{n} \sum_{i=1}^n \ell_\tau\left((x_i^T \beta - y_i)w(x_i)\right) w(x_i) + \lambda \tau \|\beta\|_1\right\}.
\end{equation}
Note that when the $x_i$'s are well-behaved (e.g., sub-Gaussian), we may set $w \equiv 1$, somewhat simplifying the analysis.

In what follows, we define $\sigmastar := \sqrt{\Var(\epsilon_i)}$. The proof of the following theorem, based on arguments developed in Loh~\cite{Loh17}, is contained in Appendix~\ref{AppThmHubOutput}:
\begin{thm*}
\label{ThmHubOutput}
Suppose $\tau \ge 3\sigma^*$ and $\lambda \ge 2c_0b' \sqrt{\frac{\log p}{n}}$, and $n \succsim k \log p$. The estimate $\betahat$ from $\ell_1$-penalized Huber regression with parameter $\tau$ satisfies
\begin{align*}
\|\betahat_\tau - \betastar\|_2 & \le C\tau \lambda \sqrt{k}, \\
\|\betahat_\tau - \betastar\|_1 & \le 4\sqrt{k} \|\betahat_\tau - \betastar\|_2,
\end{align*}
with probability at least $1 - c_1\exp(-c_2n)$.
\end{thm*}

Importantly, the theoretically optimal choice of $\lambda$ from Theorem~\ref{ThmHubOutput}, which is $2c_0 b' \sqrt{\frac{\log p}{n}}$, depends only on the \emph{known} parameter $b'$ and and a universal constant $c_0$. This is in contrast to the usual Lasso, which requires the tuning parameter $\lambda$ to be proportional to the unknown error variance $\Var(\epsilon_i)$. When $\lambda = 2c_0 b' \sqrt{\frac{\log p}{n}}$, we have the error bounds
\begin{equation*}
\|\betahat_\tau - \betastar\|_2 \le C' \tau \sqrt{\frac{k \log p}{n}}, \qquad
\|\betahat_\tau - \betastar\|_1 \le 4C' \tau k \sqrt{\frac{\log p}{n}}.
\end{equation*}

As revealed in the proof, the constant $C$ appearing in the bound of Theorem~\ref{ThmHubOutput} scales linearly with $b'$ and inversely with $\lambda_{\min}\left(\E\left[\frac{w^3(x_i)}{2} x_i x_i^T\right]\right)$. This illustrates the drawback of scaling the covariates too aggressively via the weight function $w(x)$. In particular, if it is known that the tails of covariates are well-behaved (e.g., the $x_i$'s are sub-Gaussian), one may eliminate the weight function and replace the factor by the larger quantity $\lambda_{\min}\left(\E[x_i x_i^T]\right)$.

We also comment on the requirement that $\tau \ge 3\sigma^*$. We will provide a method in the next subsection for adaptively choosing $\tau$ without prior knowledge of $\sigma^*$, with a guarantee that the estimator obtained from our procedure is at least as good as the estimator obtained by taking the theoretically optimal choice $\tau = 3\sigma^*$. However, suppose momentarily that we are able to set the Huber parameter $\tau$ equal to $3\sigma^*$, and consider for the sake of illustration that the $\epsilon_i$'s are drawn from a mixture distribution $(1-\zeta) F + \zeta G$, where $F_1$ and $G$ are both zero-mean sub-Gaussian distributions with sub-Gaussian parameters $\sigma_F^2$ and $\sigma_G^2$, respectively, and $\zeta$ is the mixing probability. Standard results on sub-Gaussian distributions imply that the mixture distribution is also sub-Gaussian, with parameter bounded by $\sigma_G^2$. Thus, Lasso theory implies that $\|\betahat_{\text{Lasso}} - \betastar\|_2 \precsim \sigma_G \sqrt{\frac{k \log p}{n}}$. On the other hand, the variance of the mixture distribution is a weighted combination of the variances of $F$ and $G$, hence is bounded by a constant multiple of $p \sigma_F^2 + (1-p) \sigma_G^2$. If $p$ is close to 1, this translates into an $\ell_2$-error bound of approximately $\|\betahat - \betastar\|_2 \precsim \sigma_F \sqrt{\frac{k \log p}{n}}$ when using an $\ell_1$-penalized Huber loss. If $\sigma_F \ll \sigma_G$, this can lead to significant gains in the estimation error.


\subsection{Lepski's method}

We now discuss Lepski's method~\cite{Lep91, Bir01}. Consider $\sigma_{\min}$ and $\sigma_{\max}$ such that $\sigma_{\min} \le \sigmastar \le \sigma_{\max}$. Let $\sigma_j = \sigma_{\min} 2^j$, and define
\begin{equation*}
\scriptJ = \{j \ge 1: \sigma_{\min} \le \sigma_j < 2 \sigma_{\max}\}.
\end{equation*}
Note that $|\scriptJ| \le \log_2\left(\frac{2\sigma_{\max}}{\sigma_{\min}}\right)$.

Let $\betahat_{(j)}$ denote the output of the regression procedure with $\tau = 3\sigma_j$, and define
\begin{multline}
\label{EqnLepJ}
j_* = \min\Bigg\{j \in \scriptJ: \; \forall i > j \text{ s.t. } i \in \scriptJ, \; \|\betahat_{(i)} - \betahat_{(j)}\|_2 \le 6C \sigma_i \sqrt{\frac{k \log p}{n}} \\
\text{and } \|\betahat_{(i)} - \betahat_{(j)}\|_1 \le 24 C \sigma_i k \sqrt{\frac{\log p}{n}} \Bigg\}.
\end{multline}
(We define $j_* = \infty$ if no such indices exist, but we will show that $j_* < \infty$, w.h.p.)
Thus, to compute $j_*$, we perform pairwise comparisons of regression estimates obtained over the gridding of the interval $[\sigma_{\min}, \sigma_{\max}]$.

Note that if our goal were simply to obtain $\ell_2$-consistency, we could apply Lepski's method where $j_*$ is simply defined with respect to comparisons involving the $\ell_2$-error. However, we will need $\ell_1$-error bounds for the one-step derivations later, so we include both deviations in the screening process here. We then have the following result:
\begin{thm*}
\label{ThmLepski}
Under the same conditions as Theorem~\ref{ThmHubOutput}, with probability at least
\begin{equation*}
1 - \log_2\left(\frac{4\sigma_{\max}}{\sigma_{\min}}\right) c \exp(-c'n),
\end{equation*}
we have
\begin{align*}
\|\betahat_{(j_*)} - \betastar\|_2 \le 18\sigmastar \sqrt{\frac{k \log p}{n}}, \qquad
\|\betahat_{(j_*)} - \betastar\|_1 \le 72 \sigmastar k \sqrt{\frac{\log p}{n}}.
\end{align*}
\end{thm*}
The proof follows from straightforward algebraic manipulations and is contained in Appendix~\ref{AppThmLepski}.


Note that Lepski's method does not correspond to a standard grid search over $\sigma$, which would be more reminiscent of the adaptive robust estimation procedures described in the introduction.
Indeed, for each candidate value of $\sigma$, we perform a type of guided comparison between different values of $\sigma$, rather than simply choosing the value of $\sigma$ that gives rise to the smallest value of some objective function. Furthermore, the output of a Lepski-type procedure does not necessarily correspond to the $\betahat$ arising from the ``optimal" choice of $\sigmastar$. Rather, we are guaranteed that the $\ell_1$- and $\ell_2$-error of our final estimate is comparable to the \emph{error of the estimator} generated using the optimal parameter. In contrast, the adaptive procedures appearing in robust statistics literature suggest a method for choosing the optimal $\sigma$ by minimizing an approximation of the variance of the estimator thus produced.


\subsection{Rough scale parameter bounds}

Our application of Lepski's method requires specifying choices of $\sigma_{\min}$ and $\sigma_{\max}$. We now describe how to select these values in a reasonable manner. By independence, we have
\begin{equation*}
\Var(y_i) = \Var(x_i^T \betastar) + \Var(\epsilon_i).
\end{equation*}
Hence, we have $(\sigmastar)^2 \le \Var(y_i)$, and we may select $\sigma_{\max}^2$ to be a rough estimate of $\Var(y_i)$.

Various estimators for population means exist that only involve weak distributional assumptions. For instance, the ``median of means" estimator takes as input $n$ i.i.d.\ observations $X_1, \dots, X_n$, and then computes the means $\{\muhat_j\}_{j=1}^K$ of the $K$ disjoint subsets of $N = \lfloor \frac{n}{K} \rfloor$ observations, for a parameter $k$. The overall estimate $\muhat_{MoM}$ is the median of the means $\{\muhat_j\}_{j=1}^K$. We have the following guarantee:
\begin{lem*}
\label{LemMeanConc}
[Lemma 2 of Bubeck et al.~\cite{BubEtal13}]
Let $0 < \delta < 1$ and $0 < \epsilon \le 1$, and $n \ge 16 \log\left(\frac{1}{\delta}\right) + 2$. Suppose $\E[X_i] = \mu$ and $\E\left[|X_i - \mu|^{1+\epsilon}\right] = v$. Let $K = \left\lfloor 8 \log\left(\frac{e^{1/8}}{\delta}\right) \wedge \frac{n}{2} \right\rfloor$. Then with probability at least $1 - \delta$,
\begin{equation*}
|\muhat_{MoM} - \mu| \le (12v)^{\frac{1}{1+\epsilon}} \left(\frac{16 \log(e^{1/8}/\delta)}{n}\right)^{\frac{\epsilon}{1+\epsilon}}.
\end{equation*}
\end{lem*}

In particular, we may take $\delta = e^{-\tilde{c} n}$, where $\tilde{c}$ is a function of $\epsilon, v$, and $\mu$, to ensure that $\mu \le 2\muhat_{MoM}$ with probability at least $1 - c \exp(-c'n)$. Consequently, we set $\sigma_{\max}^2 = 2\sigmahat_{MoM}^2$, where $\sigmahat_{MoM}^2$ is the median-of-means estimator computed from the dataset $\{y_i^2\}_{i=1}^n$. Assuming the existence of $(2+\epsilon)$-moments of $x_i$ and $\epsilon_i$, we are guaranteed that
\begin{equation*}
\left|\sigmahat_{MoM}^2 - \E[y_i^2]\right| \le \frac{\E[y_i^2]}{2},
\end{equation*}
with probability at least $1 - c \exp(-c'n)$, so
\begin{equation*}
\sigma_{\max}^2 = 2\sigmahat_{MoM}^2 \ge \E[y_i^2] \ge \Var(y_i) \ge \Var(\epsilon_i) = (\sigmastar)^2
\end{equation*}
and
\begin{equation*}
\sigma_{\max}^2 \le \frac{3}{2} \Var(y_i) = \frac{3}{2}\left((\betastar)^T \Sigma_x \betastar + (\sigmastar)^2\right).
\end{equation*}

We now turn to the problem of choosing $\sigma_{\min}$. Consider the choice $\sigma_{\min} = \frac{\sigma_{\max}}{2^M}$, for some integer $M$. Let $\betahat$ be the final output of Lepski's method. By Theorem~\ref{ThmLepski}, we have
\begin{equation*}
\|\betahat - \betastar\|_2 \le 18C\sigmastar \sqrt{\frac{k \log p}{n}},
\end{equation*}
with probability at least $1 - cM\exp(-c'n)$.
Thus, we may see the selection of $M$ as a tradeoff between computation and accuracy: A larger value of $M$ ensures an (exponentially) tighter bound on $\|\betahat - \betastar\|_2$, but the number of $\ell_1$-penalized Huber regression objectives to be optimized also increases linearly with $M$. We have the following result:
\begin{thm*}
\label{ThmAdap}
Suppose Lepski's method is performed on the $\ell_1$-penalized Huber problems with $\sigma_{\max}^2$ equal to the median-of-means estimator of $\Var(y_i)$ and $\sigma_{\min} = \frac{\sigma_{\max}}{2^M}$. If
\begin{equation}
\label{EqnWeird}
\frac{3}{2^{M/2 + 1}} \left((\betastar)^T \Sigma_x \betastar + (\sigmastar)^2\right) \le (\sigmastar)^2,
\end{equation}
we have
\begin{align}
\|\betahat - \betastar\|_2 & \le 18 C \sigmastar \sqrt{\frac{k \log p}{n}}, \qquad \text{and} \label{EqnEll2} \\
\|\betahat - \betastar\|_1 & \le 72 C\sigmastar k \sqrt{\frac{\log p}{n}}, \label{EqnEll1}
\end{align}
with probability at least $1-cM \exp(-c'n)$.
\end{thm*}

Note that if $M = o(c'n)$, the bounds~\eqref{EqnEll2} and~\eqref{EqnEll1} in Theorem~\ref{ThmAdap} hold w.h.p. In particular, note that if $M$ grows with $n$ and $\lambda_{\max}(\Sigma_x)$ and $\|\betastar\|_2$ are bounded, then inequality~\eqref{EqnWeird} holds for all sufficiently large $n$. Note also that some knowledge of the curvature of the covariate distribution (i.e., maximum eigenvalue of $\Sigma_x$) can be helpful in determining the choice of $M$ necessary for inequality~\eqref{EqnWeird} to be satisfied.

\subsection{Practical considerations}

In practice, we need to have an explicit value for $C$ in order to apply Lepski's method. As noted earlier, the constant $C$ appearing in our bounds depends on universal constants, the choice $\tau$ of the parameter used for the robust loss function, and distributional properties of $x_i$. The last point is somewhat unsatisfactory. However, in practical applications, we might imagine having numerous samples of the $x_i$'s available, from which we might be able to estimate a lower bound on $\lambda_{\min}\left(\E\left[\frac{w^3(x_i)}{2} x_i x_i^T\right]\right)$. 
Importantly, we emphasize that our proposed method does \emph{not} require any information about the distribution of the $\epsilon_i$'s, which we would not be accessible without a good initial estimate of $\betastar$ in practice.



\section{One-step estimators}
\label{SecEfficient}

Although we have established the consistency of our estimators under rather weak distributional assumptions on $x_i$ and $\epsilon_i$, it is well-known that the presence of the weight function $w(x)$ leads to poor efficiency. We now address the problem of improving efficiency by studying one-step modifications of the estimators proposed in the previous section.

\subsection{Main results}
\label{SecNormality}

The theory of $M$-estimation from classical robust statistics recommends one-step estimators for improved efficiency \cite{Rou84, JurPor87, SimEtal92, GerYoh02}.
Recent results in high-dimensional inference have led to theoretical derivations based on similar types of one-step estimators, which we analyze here.

Consider a differentiable score function $\psi$, and let $A(\psi) = \E[\psi'(\epsilon_i)]$. Also define the empirical estimate $\Ahat(\psi) = \frac{1}{n} \sum_{i=1}^n \psi'(y_i - x_i^T \betahat)$. Following Bickel~\cite{Bic75}, we then define the one-step estimator
\begin{equation*}
\bhat = \betahat + \frac{\Thetahat}{\Ahat(\psi)} \cdot \frac{1}{n} \sum_{i=1}^n \psi(y_i - x_i^T \betahat) x_i,
\end{equation*}
where $\Thetahat$ is a suitable estimate of $\Theta_x = \Sigma_x^{-1}$, to be described in the sequel.

Since the scale is unknown, we will plug an estimate $\sigmahat$ into the score function $\psi_\sigma(t) = \psi\left(\frac{t}{\sigma}\right)$ and use the corresponding one-step estimator
\begin{align}
\label{EqnOneStepPsi}
\bhat_\psi & := \betahat + \frac{\Thetahat}{\Ahat(\psi_{\sigmahat})} \cdot \frac{1}{n} \sum_{i=1}^n \psi_{\sigmahat}(y_i - x_i^T \betahat) x_i \notag \\
& = \betahat + \frac{\Thetahat}{\Ahat(\psi_{\sigmahat})} \cdot \frac{1}{n} \sum_{i=1}^n \psi\left(\frac{y_i - x_i^T \betahat}{\sigmahat}\right) x_i,
\end{align}
where
\begin{equation}
\label{EqnAhat}
\Ahat(\psi_{\sigmahat}) = \frac{1}{n} \sum_{i=1}^n \psi'_{\sigmahat}(y_i - x_i^T \betahat) = \frac{1}{n} \sum_{i=1}^n \frac{1}{\sigmahat} \psi'\left(\frac{y_i - x_i^T \betahat}{\sigmahat}\right),
\end{equation}
and the scale estimate $\sigmahat$ is obtained from the consistent regression parameter estimate via $\sigmahat = \frac{1}{n} \sum_{i=1}^n (y_i - x_i^T \betahat_{\Lep})^2$. For ease of notation, we will redefine the term $\Ahat(\psi)$ to be equal to the expression~\eqref{EqnAhat}, and let $A(\psi) = \E\left[\frac{1}{\sigmastar}\psi'\left(\frac{\epsilon_i}{\sigmastar}\right)\right]$.

\begin{example*}
The choice $\psi = \frac{-f'}{f}$, where $f$ is the (twice-differentiable) density of $\frac{\epsilon_i}{\sigmastar}$, will play a prominent role in our analysis. This corresponds to the derivative of the negative log likelihood function. In the case when $\frac{\epsilon_i}{\sigmastar} \sim N(0, 1)$, we then have $\psi_\sigma(t) = t$ and $\psi_\sigma'(t) = 1$, in which case
\begin{equation*}
\bhat_\psi = \betahat + \frac{\Thetahat X^T (y - X\betahat)}{n},
\end{equation*}
which is the ``debiased Lasso" \cite{vanEtal14, JavMon14JMLR, CaiGuo17, JanVan18, YuEtal18}. However, in that line of work, $\betahat$ is always taken to be the output of the usual MLE-based objective, whereas we take $\betahat$ to be a more general robust high-dimensional estimator with guaranteed statistical consistency properties even when the $\epsilon_i$'s are non-sub-Gaussian.
\end{example*}

We now discuss how to obtain a suitable estimate $\Thetahat$ of $\Theta_x$. Note that Bickel~\cite{Bic75} proposes to use $\Thetahat = \Sigmahat^{-1}$, where $\Sigmahat = \frac{X^TX}{n}$; however, when $p > n$, the matrix $\Sigmahat$ is not invertible. We instead choose $\Thetahat$ to be the graphical Lasso estimator \cite{YuaLin07, FriEtal08}, obtained by solving the following convex optimization program:
\begin{equation}
\label{EqnGLasso}
\Thetahat \in \arg\min_{\Theta \succeq 0} \left\{\tr\left(\Theta^T \Sigmahat\right) - \log\det(\Theta) + \lambda \sum_{i \neq j} |\Theta_{ij}|\right\}.
\end{equation}


We now derive the limiting distribution of the one-step estimator. Our arguments involve Taylor expansions of the function $\psi$, so for simplicity, we assume that $\psi$ is thrice-differentiable. Extensions to cases where $\psi$ does not satisfy the differentiability criterion (e.g., corresponding to the Huber loss function) may be derived via more careful arguments, but we omit the details here. Our main result is the following:
\begin{thm*}
\label{ThmAsympNorm}
Suppose $n \succsim k^2 \log^3 p$ and
\begin{align}
|\Ahat(\psi) - A(\psi)| & = \order\left(\frac{k \log p}{\sqrt{n}}\right), \label{EqnACond} \\
\opnorm{\Thetahat - \Thetastar}_1 & = \order\left(k \sqrt{\frac{\log p}{n}}\right) \label{EqnThetaCond}.
\end{align}
Let $P_J$ denote the projection onto any set of $m = |J|$ coordinates of fixed dimension, and suppose we have the convergence in distribution
\begin{equation}
\label{EqnConvDist}
P_J \cdot \frac{\Theta_x}{A(\psi)} \cdot \frac{1}{\sqrt{n}} \sum_{i=1}^n \psi\left(\frac{\epsilon_i}{\sigmastar}\right) x_i \stackrel{d}{\longrightarrow} N\left(0, \; \frac{\E[\psi^2(\epsilon_i/\sigmastar)]}{A^2(\psi)} \cdot (\Theta_x)_{JJ} \right).
\end{equation}
Then the one-step estimator satisfies
\begin{equation*}
\sqrt{n} P_J(\bhat_{\psi} - \betastar) \stackrel{d}{\longrightarrow} N\left(0, \; \frac{\E[\psi^2(\epsilon_i/\sigmastar)]}{A^2(\psi)} \cdot (\Theta_x)_{JJ} \right).
\end{equation*}
\end{thm*}
The proof of Theorem~\ref{ThmAsympNorm} is contained in Appendix~\ref{AppThmAsympNorm}.

Altogether, we conclude that the limiting distribution of the high-dimensional estimator, restricted to $m$ coordinates, agrees with the result of Bickel~\cite{Bic75} for low-dimensional robust $M$-estimators.



\subsection{Probabilistic conditions}
\label{SecProbCond}

We now present some distributional assumptions under which the conditions~\eqref{EqnACond}, \eqref{EqnThetaCond}, and~\eqref{EqnConvDist} hold w.h.p. Importantly, we need stronger assumptions to guarantee asymptotic normality of the one-step estimator than the conditions we have imposed so far to prove consistency. We make the following assumptions on the distribution of the covariates:
\begin{assumption} Assume that
\begin{itemize}
\item[(A1)] $\min_{1 \le j, k \le p} \Var(x_{ij} x_{ik}) \ge c_1$,
\item[(A2)] $\max_{1 \le j, k \le p} \E\left[\exp\left(\frac{|x_{ij} x_{ik} - (\Sigma_x)_{jk}|}{C_1}\right)\right] \le 2$,
\item[(A3)] $\max_{1 \le j, k \le p} \Var(x_{ij} x_{ik}) \le \sigma^2_{xx}$,
\item[(A4)] $\max_{1 \le j \le p} \E\left[\left(e_j^T \Theta_x x_i\right)^3\right] < \infty$.
\end{itemize}
\end{assumption}
As discussed in Chernuzhukov et al.~\cite{CheEtal13}, these conditions are certainly satisfied when the $x_{ij}$'s are sub-Gaussian (hence, $x_{ij} x_{ik}$ is sub-exponential), but the conditions are designed to hold for broader classes of distributions, as well. In particular, we essentially only require a bound on the mgf of the $x_{ij}$'s at one point. We also assume that the following assumptions on the additive errors, where for simplicity, we reuse the constant $C_1$ appearing in (A2):
\begin{assumption} Assume that
\begin{itemize}
\item[(B1)] $\E[\epsilon_i^2] < \infty$, $\E\left[\exp\left(\frac{\epsilon_i^2}{C_1}\right)\right] \le 2$,
\item[(B2)] $\E\left[\psi^3\left(\frac{\epsilon_i}{\sigmastar}\right)\right] < \infty$, $\E\left[\exp\left(\frac{\psi^2(\epsilon_i/\sigmastar)}{C_1}\right)\right]  \le 2$, $\E\left[\exp\left(\frac{\left(\psi''(\epsilon_i/\sigmastar)\right)^2}{C_1}\right)\right]  \le 2$,
\item[(B3)] $\E\left[\exp\left(\frac{\psi^2(\epsilon_i/\sigmastar)\left(\psi'(\epsilon_i/\sigmastar)\right)^2}{C_1}\right)\right]  \le 2$, $\E\left[\exp\left(\left(\frac{\epsilon_i}{\sigmastar}\right)^2\psi^2\left(\frac{\epsilon_i}{\sigmastar}\right)/C_1\right)\right] \le 2$.
\end{itemize}
\end{assumption}
The first set of results on consistency will only require assumptions (B1)--(B2), whereas our results on constructing confidence intervals will also require (B3). Note that (B3) is relatively stronger, since (B2) allows the heaviness of the tails of $\epsilon_i$ to be alleviated using a well-behaved $\psi$ function (for instance, if $\psi$ is bounded, assumption (B2) will be trivially satisfied).

We have the following result, proved in Appendix~\ref{AppLemApsi}:
\begin{lem*}
\label{LemApsi}
Suppose assumptions (A1)--(A4) and (B1)--(B2) hold, and suppose $\|\psi'\|_\infty, \|\psi^{(3)}\|_\infty < \infty$. Then conditions~\eqref{EqnACond} and~\eqref{EqnConvDist} hold,
with probability at least $1 - c\exp(-c'n)$.
\end{lem*}

We now turn to the inverse covariance estimator. Suppose $\Sigmastar$ satisfies the $\alpha$-incoherence condition, defined by
\begin{equation}
\label{EqnIncoherence}
\max_{e \in S^c} \opnorm{\Gammastar_{eS} (\Gammastar_{SS})^{-1}}_1 \le 1 - \alpha,
\end{equation}
where $\alpha \in (0,1]$, and we denote $\Gammastar := \Sigmastar \otimes \Sigmastar$ and $S = \supp(\Thetastar)$. We also denote $\kappa_{\Sigmastar} := \opnorm{\Sigmastar}_1$ and $\kappa_{\Gammastar} := \opnorm{(\Gammastar_{SS})^{-1}}_1$.

Combining Lemma~\ref{LemCovX} in Appendix~\ref{AppUseful} with standard derivations for the graphical Lasso~\cite{RavEtal11} yields the following result:
\begin{lem*}
\label{LemGLasso}
Suppose assumptions (A1)--(A4) hold. Also suppose $\Theta^*$ satisfies the $\alpha$-incoherence condition~\eqref{EqnIncoherence} and the regularization parameter satisfies
\begin{equation*}
\frac{c_0 \sigma_{xx}}{\alpha} \sqrt{\frac{\log p}{n}} \le \lambda \le \frac{1}{6\kappa_{\Gammastar}k} \left(\frac{\alpha}{8} + 1\right)^{-1} \min\left\{\frac{1}{\kappa_{\Sigmastar}}, \; \frac{1}{\kappa_{\Sigmastar}^3 \kappa_{\Gammastar}}, \; \frac{\alpha(\alpha/8 + 1)^{-1}}{8\kappa^3_{\Sigmastar} \kappa_{\Gammastar}}\right\}.
\end{equation*}
With probability at least $1 - Cn^{-c} + \exp(-c'\log p)$, the graphical Lasso estimator~\eqref{EqnGLasso} satisfies $\supp(\Thetahat) \subseteq \supp(\Thetastar)$, and
\begin{equation*}
\|\Thetahat - \Thetastar\|_{\max} \le 2 \opnorm{(\Gammastar_{SS})^{-1}}_1 \left(1+\frac{\alpha}{8}\right)\lambda.
\end{equation*}
In particular, if each row of $\Thetastar$ is $k$-sparse, we also have the bound
\begin{equation*}
\opnorm{\Thetahat - \Thetastar}_1 \le 2 \opnorm{(\Gammastar_{SS})^{-1}}_1 \left(1+\frac{\alpha}{8}\right)\lambda k.
\end{equation*}
\end{lem*}
The proof of Lemma~\ref{LemGLasso} is contained in Appendix~\ref{AppLemGLasso}.

Note that simply applying Theorem 1 in Ravikumar et al.~\cite{RavEtal11} would produce a weaker result than we want, since the concentration result in Lemma~\ref{LemCovX} would fall into the category of ``polynomial-type tails," thus yielding a suboptimal sample size requirement. Instead, we derive a statistical error guarantee suitable for our setting, building upon some of the key lemmas in Ravikumar et al.~\cite{RavEtal11}.

\subsection{Semiparametric efficiency}
\label{SecSemiEff}

To make the notions of increased efficiency more precise, we now analyze the one-step estimator $\bhat$ from the point of view of semiparametric efficiency. A review of relevant background material is contained in Appendix~\ref{AppSemi}, and the main result is Theorem~\ref{ThmSemiEff}, which states that a lower bound on the variance of any semiparametrically efficient estimator for the semiparametric regression model
\begin{equation*}
y_i = x_i^T \beta_0 + g_0(v_i) + \epsilon_i,
\end{equation*}
where $f$ is the (known) density of $\epsilon_i$, is
\begin{equation*}
\widebar{V} = \left(\E\left[\left(\frac{f'(\epsilon)}{f(\epsilon)}\right)^2\right] \cdot \E\left[(x - \E[x|v])(x-\E[x|v])^T\right]\right)^{-1}.
\end{equation*}

For a fixed set of indices $J \subseteq \{1, \dots, p\}$, we partition the linear model as
\begin{equation*}
y_i = (x_i)_J^T \betastar_J + (x_i)_{J^c}^T \betastar_{J^c} + \epsilon_i
\end{equation*}
and consider it as a subclass of the semiparametric regression model
\begin{equation}
\label{EqnSemiMain}
y_i = (x_i)_J^T \betastar_J + g_0\left((x_i)_{J^c}\right) + \epsilon_i.
\end{equation}
We then have the following result, proved in Appendix~\ref{AppThmEfficient}:
\begin{thm*}
\label{ThmEfficient}
Suppose we have i.i.d.\ observations from the linear model~\eqref{EqnLinear}. Under the assumptions of the previous theorems, the one-step estimator $(\bhat_\psi)_J$ with $\psi = \frac{-f'}{f}$, where $f$ is the pdf of the distribution of $\frac{\epsilon_i}{\sigmastar}$, is semiparametrically efficient for the model~\eqref{EqnSemiMain}.
\end{thm*}


\begin{rem}
\label{RemDebiased}
We may compare this result with Section 3 of van de Geer et al.~\cite{vanEtal14}, in which Lasso debiasing results are derived for general convex loss functions. Translating to the linear model with i.i.d.\ (but not necessarily Gaussian) additive errors, the proposed one-step estimator takes the form
\begin{equation}
\label{EqnOneStepRho}
\bhat_\rho = \betahat + \Thetahat_\rho \cdot \frac{1}{n} \sum_{i=1}^n \rho'(y_i - x_i^T \betahat) x_i,
\end{equation}
where $\betahat$ is the solution to the $\ell_1$-penalized program
\begin{equation}
\label{EqnRhoObj}
\betahat \in \arg\min_\beta \left\{\frac{1}{n} \sum_{i=1}^n \rho(y_i - x_i^T \beta) + \lambda \|\beta\|_1\right\},
\end{equation}
and $\rho$ is assumed to be a smooth convex function. Furthermore, $\Thetahat_\rho$ is defined to be a sparse approximate inverse of the matrix $\frac{1}{n} \sum_{i=1}^n \rho''(y_i - x_i^T \betahat) x_i x_i^T$.

Although clear similarities exist between the one-step estimator~\eqref{EqnOneStepRho} and the expression~\eqref{EqnOneStepPsi}, with $\rho'$ taking the place of $\psi$, the one-step estimator~\eqref{EqnOneStepRho} is only guaranteed to be asymptotically normal when standardized appropriately. Furthermore, note that the $M$-estimator~\eqref{EqnRhoObj} is not designed to be robust to contaminated covariates, and in order to obtain appropriate error bounds, much stronger assumptions must be made on the distribution of the $x_i$'s.
\end{rem}

Finally, we note that another notion of semiparametric efficiency was recently studied in Jankova and van de Geer~\cite{JanVan16}, involving a more complicated infinite-dimensional model that is allowed to change with $n$. It was shown that when $\Theta_x$ is a sparse matrix, the same bounds may be established for semiparametric efficiency; however, van de Geer~\cite{van17} showed that without the sparsity condition, the variance of an efficient estimator may in fact be lower. We suspect that these notions could also be adapted to the setting of robust regression estimators discussed in our paper, but such derivations are beyond the scope of our present work.

The notions of efficiency we have just described should also be contrasted with the discussion of efficiency contained in Loh~\cite{Loh17}. Importantly, our present results do not require any conditions for correct support recovery, which were rather strong requirements imposed in the theory of the aforementioned paper. Furthermore, by using a one-step estimator, we do not require a second subgradient optimization routine performed on a nonconvex objective function in order to achieve efficiency, since a one-step modification of the global optimum of the convex surrogate is sufficient for our purposes.


\subsection{Confidence intervals}
\label{SecConf}

Our results from Section~\ref{SecNormality} in fact allow us to derive confidence intervals with the correct asymptotic coverage, which we briefly describe here. Furthermore, the result of Section~\ref{SecSemiEff} provides a type of ``optimality" guarantee for the size of the confidence region. We again consider a fixed subset $J \subseteq \{1, \dots, p\}$, where $|J| = m$.

Note that our work improves over previous literature in removing the assumption of sub-Gaussianity on either the $x_i$'s or $\epsilon_i$'s. Hence, in settings relevant to robust linear regression, we are also able to derive valid confidence intervals for the regression vector via a type of one-step estimator based on a consistent initial estimator.
For an error probability $\alpha \in (0,1)$, we write $\scriptB_{\alpha, J}$ to denote the subset of $\real^J$ corresponding to the direct product of $m$ intervals of the form
\begin{equation*}
\left[-\Phi^{-1} \left(\frac{1+(1-\alpha)^{1/m}}{2}\right), \; \Phi^{-1} \left(\frac{1+(1-\alpha)^{1/m}}{2}\right)\right],
\end{equation*}
where $\Phi$ is the cdf of a standard normal random variable. In particular, if $Z \sim N(0, I_m)$ is an $m$-dimensional Gaussian random vector with i.i.d.\ standard normal components, we have
\begin{equation}
\label{EqnScaleBox}
\mprob(Z \in \scriptB_{\alpha, J}) = \left(1-2\left(1 - \frac{1+(1-\alpha)^{1/m}}{2}\right)\right)^m = 1-\alpha.
\end{equation}

We have the following main result, proved in Appendix~\ref{AppThmConfReg}:
\begin{thm*}
\label{ThmConfReg}
Let $|J| = m$ be a fixed set of constant cardinality. In addition to the assumptions of Theorem~\ref{ThmAsympNorm}, suppose assumption (B3) holds, and $\|(\psi^2)''\|_\infty < \infty$. An asymptotically valid $(1-\alpha)$-confidence region for the projection $\betastar_J$ of the regression vector onto $J$ is given by
\begin{equation}
\label{EqnConfReg}
P_J \bhat_\psi + \frac{1}{\sqrt{n}} \cdot \frac{\sqrt{\frac{1}{n} \sum_{i=1}^n \psi^2\left((y_i - x_i^T \betahat)/\sigmahat\right)}}{\Ahat(\psi)} \cdot \left(\Thetahat_{JJ}\right)^{1/2} \scriptB_{\alpha, J}.
\end{equation}
Note that the region~\eqref{EqnConfReg} is a (pointwise) linear transformation of $\scriptB_{\alpha,J}$.
\end{thm*}

In the case $m = 1$, the confidence region for a fixed coordinate $j$ reduces to the interval
\begin{equation*}
(\bhat_\psi)_j \pm \frac{1}{\sqrt{n}} \cdot \frac{\sqrt{\frac{1}{n} \sum_{i=1}^n \psi^2\left((y_i - x_i^T \betahat)/\sigmahat\right)}}{\Ahat(\psi)} \cdot \sqrt{\Thetahat_{jj}} \cdot \Phi^{-1}\left(1-\frac{\alpha}{2}\right).
\end{equation*}
Note that as in Javanmard and Montanari~\cite{JavMon14JMLR}, the set $\scriptB_{\alpha, J}$ could be replaced with any other set of measure $1-\alpha$ under an $m$-dimensional standard normal distribution.

Note that Theorem~\ref{ThmConfReg} is a result that holds for \emph{any} choice of score functions $\psi$, not necessarily corresponding to the score function of the true pdf. Importantly,  we can construct valid confidence intervals without needing to know the true distribution of the $\epsilon_i$'s. However, in order to construct \emph{optimal} intervals, we would need to use the correct $\psi$ function corresponding to the distribution. Further note that the condition $\|(\psi^2)''\|_\infty < \infty$ will certainly hold when $\psi, \psi'$, and $\psi''$ are all bounded, and will also hold when $\psi(t) = t$ is the score function of the standard normal pdf.

\begin{rem}
As mentioned in Remark~\ref{RemDebiased}, our recipe for constructing confidence intervals resembles the proposal of van de Geer et al.~\cite{vanEtal14}. However, the key difference is that the vanilla Lasso estimator would in general not achieve the correct rates of consistency in order for the confidence intervals to be asymptotically valid for the prescribed sample size scaling. Similarly, Javanmard and Montanari~\cite{JavMon14JMLR} include a section in their paper discussing how to construct confidence intervals in the case of non-Gaussian noise; however, again, they assume that the noise and covariance distributions are sufficiently well-behaved to guarantee fast convergence of the initial Lasso estimator.
\end{rem}

Finally, it is worth discussing the relationship between our proposed method and the robust inference procedures studied in classical robust statistics. These include robust Wald-type and likelihood-ratio type tests \cite{Ron82, HamEtal11}, which are more generally applicable to hypothesis testing scenarios involving linear combinations of predictors. Our method resembles Wald-type tests in the sense that they are constructed with respect to a robust $M$-estimator, and also include robust estimates of the (inverse) covariance---however, our results are primarily designed for hypothesis testing of single coordinates. It is an interesting open question to see if analogs of the robust Wald-type or $\tau$-tests \cite{Ron82} could be derived in the high-dimensional setting. It is plausible that such tests exist using an initial $M$-estimator such as the regression estimator introduced in this paper (cf.\ van de Geer and Stucky~\cite{vanStu16} and Sur et al.~\cite{SurEtal17} for some theory in the non-robust setting).


\section{Simulations}

We now report the result of experiments that we performed to validate our theoretical predictions.


\subsection{Summary of procedure}

We first briefly summarize the steps of the robust regression procedure.

\begin{enumerate}
\item Compute rough lower and upper bounds on the scale, using the median of means estimator with tolerance $\delta$ and $K = \left \lfloor 8 \log\left(\frac{e^{1/8}}{\delta}\right) \wedge \frac{n}{2} \right \rfloor$ groups: $\sigma_{\max}^2 = 2 \sigma^2_{MoM}$, and $\sigma_{\min} = \frac{\sigma_{\max}}{2^{\sqrt{n}}}$.
\item Compute the $\ell_1$-penalized Huber $M$-estimator $\betahat_\tau$ for all $\tau$ in a grid of values from $\sigma_{\min} = \frac{\sigma_{\max}}{2^M}$ to $\sigma_{\max}$, according to the program~\eqref{EqnHuberReg}.
\item Use Lepski's method to adaptively choose $\betahat_{\Lep}$: $\betahat_{\Lep} = \betahat_{(j^*)}$, according to the rule~\eqref{EqnLepJ}.
\item Use one-step estimation to improve efficiency, according to equation~\eqref{EqnOneStepPsi}, with $\betahat = \betahat_{\Lep}$ and $\Thetahat$ from the graphical Lasso.
\end{enumerate}

\textbf{Composite gradient descent:} In order to obtain the estimators $\betahat_\tau$ in the second step above, we employ the composite gradient descent algorithm, which has fast rates of convergence for convex functions~\cite{Nes07}. Specifically, the updates are
\begin{align*}
\betahat^{t+1} & \in \arg\min_{\beta} \left\{\Loss_n(\beta^t) +\inprod{\Loss_n(\beta^t)}{\beta - \beta^t} + \frac{\eta}{2} \|\beta - \beta^t\|_2^2 + \lambda \tau \|\beta\|_1\right\} \\
& = \arg\min_\beta \left\{\frac{1}{2} \left\|\beta - \left(\beta^t - \frac{1}{\eta} \nabla \Loss_n(\beta^t)\right)\right\|_2^2 + \frac{\lambda \tau}{\eta} \|\beta\|_1\right\} \\
& = S_{\lambda \tau/\eta}\left(\beta^t - \frac{1}{\eta} \nabla \Loss_n(\beta^t)\right),
\end{align*}
where $S_{\lambda \tau/\eta}(\beta)$ is the soft-thresholding operator defined componentwise according to
\begin{equation*}
S^j_{\lambda \tau/\eta}(\beta) = \sign(\beta_j)\left(|\beta_j| - \frac{\lambda \tau}{\eta}\right)_+.
\end{equation*}
Note also that
\begin{equation*}
\nabla \Loss_n(\beta) = \frac{1}{n} \sum_{i=1}^n \ell'_\tau\left((x_i^T \beta - y_i) w(x_i)\right) w^2(x_i) x_i.
\end{equation*}


\subsection{Synthetic data}

We first ran experiments involving synthetic data to check the validity of our theory. In particular, the simulation results confirmed that our estimator is (a) consistent and (b) efficient. We provide simulation results for two different scenarios:
\begin{itemize}
\item[(i)] Additive errors $\epsilon_i$ are drawn from a heavy-tailed distribution, but $x_i$'s have a sub-Gaussian distribution.
\item[(ii)] Both $\epsilon_i$'s and $x_i$'s are drawn from heavy-tailed distributions.
\end{itemize}
In case (i), we generated the $x_i$'s from a standard normal distribution. The $\epsilon_i$'s were generated from a $t$-distribution with 3 degrees of freedom, to make the variance finite. We then scaled the additive errors by 0.01. In case (ii), we generated both $x_i$'s and $\epsilon_i$'s from a $t$-distribution with 3 degrees of freedom, and again scaled the additive errors by 0.01.

In each case, we will also approximated the variance of our estimator based on repeated draws of the data, and compared the results with and without the one-step procedure. Our theory predicts that the one-step estimator $\bhat$ has a smaller asymptotic variance than the estimate $\betahat_{\Lep}$ obtained without the final step. However, the estimate $\betahat_{\Lep}$ should always lead to a consistent estimator, provided a proper weighting function $w$ is used in the Huber $M$-estimator in the case of (ii).

For the one-step estimator, we need to compute $\hat{A}$, which depends on $\psi'$. Recall that the pdf of a $t$-distribution with three degrees of freedom and scale parameter 1 is equal to
\begin{equation*}
f(t) = \frac{2}{\pi \sqrt{3}} \frac{1}{\left(1 + \frac{t^2}{3}\right)^2} = \frac{6\sqrt{3}}{\pi(3+t^2)^2}.
\end{equation*}
Then
\begin{equation*}
f'(t) = \frac{-24\sqrt{3}}{\pi} \cdot \frac{t}{(3+t^2)^3},
\end{equation*}
from which we may compute
\begin{equation*}
\psi(t) = \frac{-f'(t)}{f(t)} = \frac{4t}{3+t^2}, \quad \text{and} \quad \psi'(t) = \frac{-4t^2+12}{(3+t^2)^2}.
\end{equation*}

Finally, we set the error tolerance $\delta = 0.05$ for the MoM estimator, and took $\sigma_{\min} = \frac{\sigma_{\max}}{2^{2n^{1/3}}}$ for the Lepski gridding. We took $b = 1$ and $B = I_p$, and $\lambda = 0.005b\sqrt{\frac{\log p}{n}}$ for the penalized Huber estimators, and used $C = 20$ when adaptively choosing the regression parameters in Lepski's method.

The plots in Figures~\ref{FigL2Var1} and~\ref{FigL2Var2} show the $\ell_2$-error of the estimator $\betahat_{\Lep}$ computed via Lepski's method, in comparison with the error $\|\bhat_\psi - \betastar\|_2$ of the one-step correction. We also compare the variance of individual coordinates. As seen in the plots of Figure~\ref{FigL2Var1}, the $\ell_2$-error of the one-step estimator is comparable (and generally slightly smaller) than the $\ell_2$-error of the initial Lepski estimator, and both estimators appear to be consistent. Furthermore, the variance of the estimator is generally reduced after the one-step correction. Similar behavior is seen in Figure~\ref{FigL2Var2}.

\begin{figure}
\begin{center}
\begin{tabular}{cc}
\includegraphics[width=0.48\textwidth]{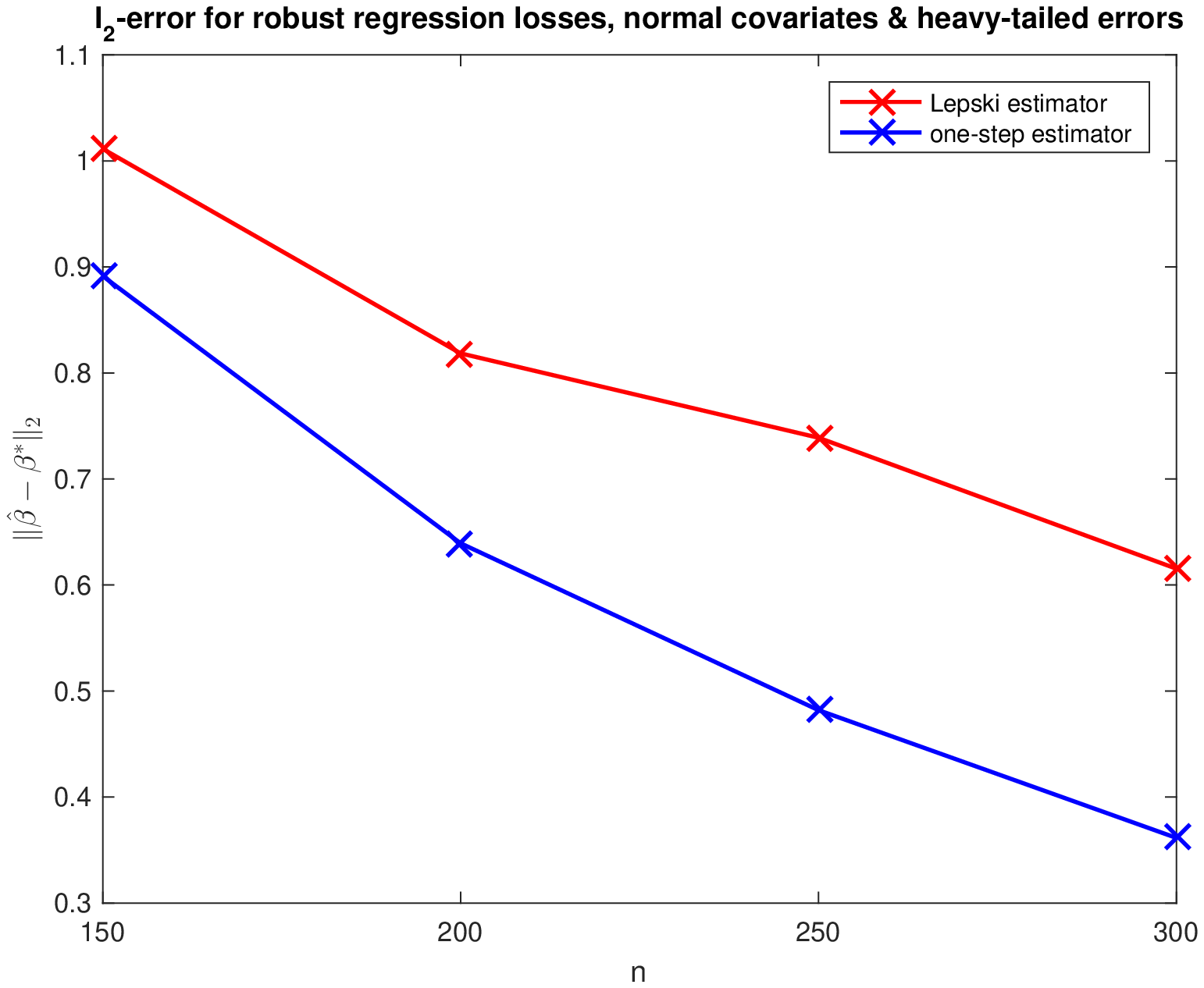} & \includegraphics[width=0.48\textwidth]{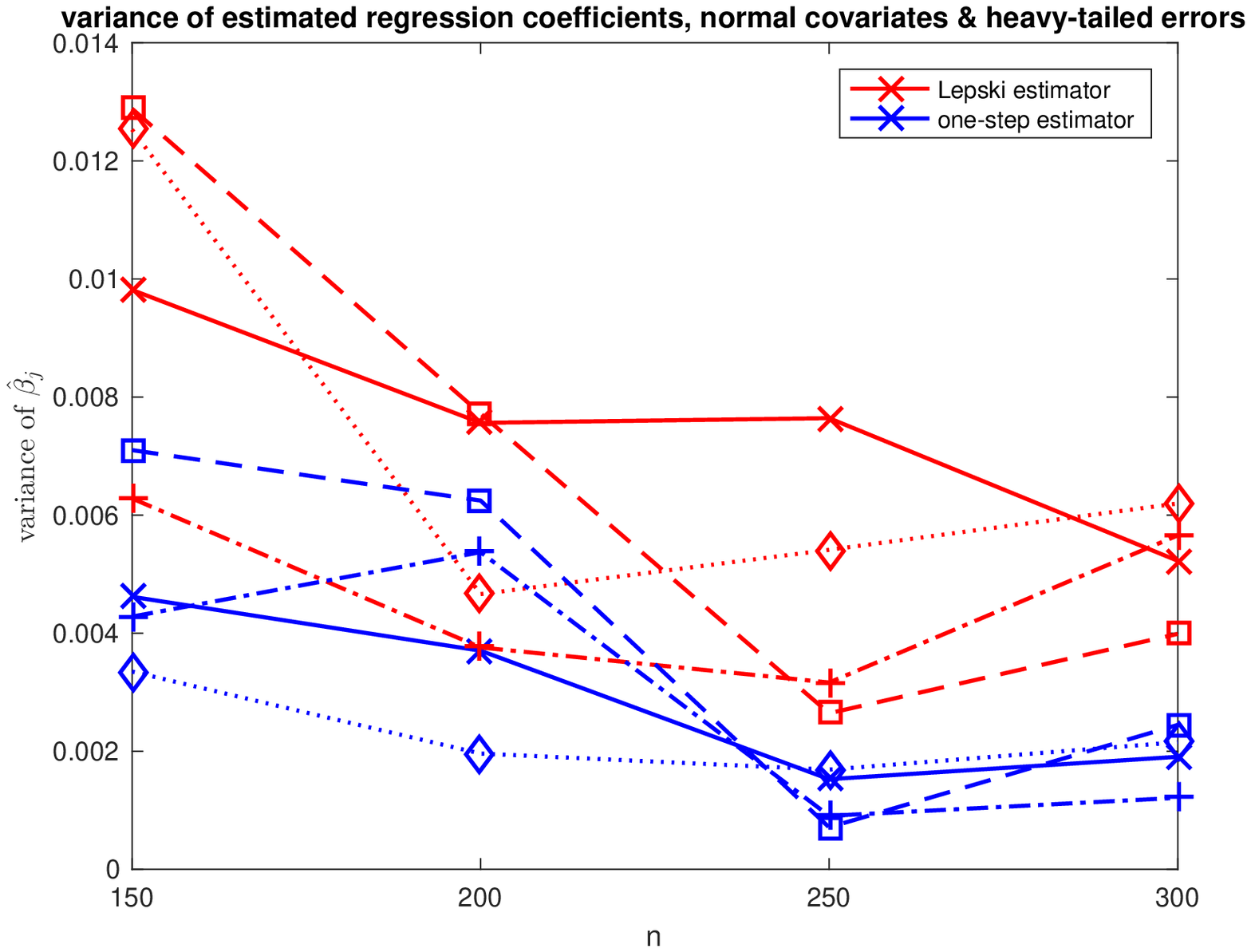} \\
(a) & (b)
\end{tabular}
\caption{Plots comparing $\ell_2$-error and variance of estimators obtained via Lepski's method (red) and Lepski's method followed by a one-step correction (blue), when $p = 200$ and $k = 4$. Covariates were generated from a multivariate normal distribution with identity covariance, and errors were generated from a $t$-distribution with 3 degrees of freedom. Panel (a) shows the error $\|\betahat - \betastar\|_2$, averaged over 10 trials. Panel (b) shows the empirical variance of $\betahat_j$, for each of the four nonzero regression coefficients, computed with respect to 10 trials. The coefficients are distinguished in the figure using different line markings.}
\label{FigL2Var1}
\end{center}
\end{figure}

\begin{figure}
\begin{center}
\begin{tabular}{cc}
\includegraphics[width=0.48\textwidth]{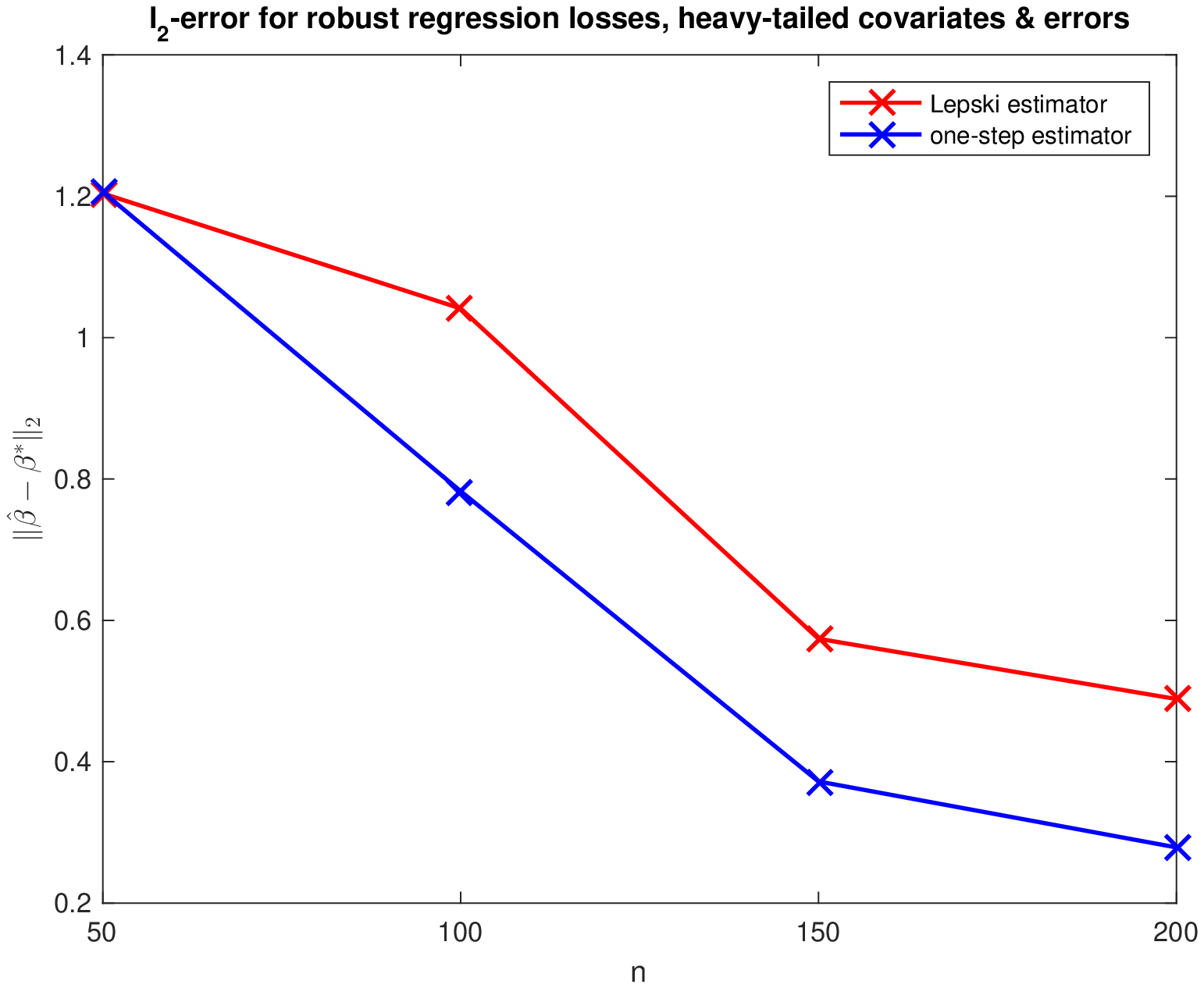} & \includegraphics[width=0.48\textwidth]{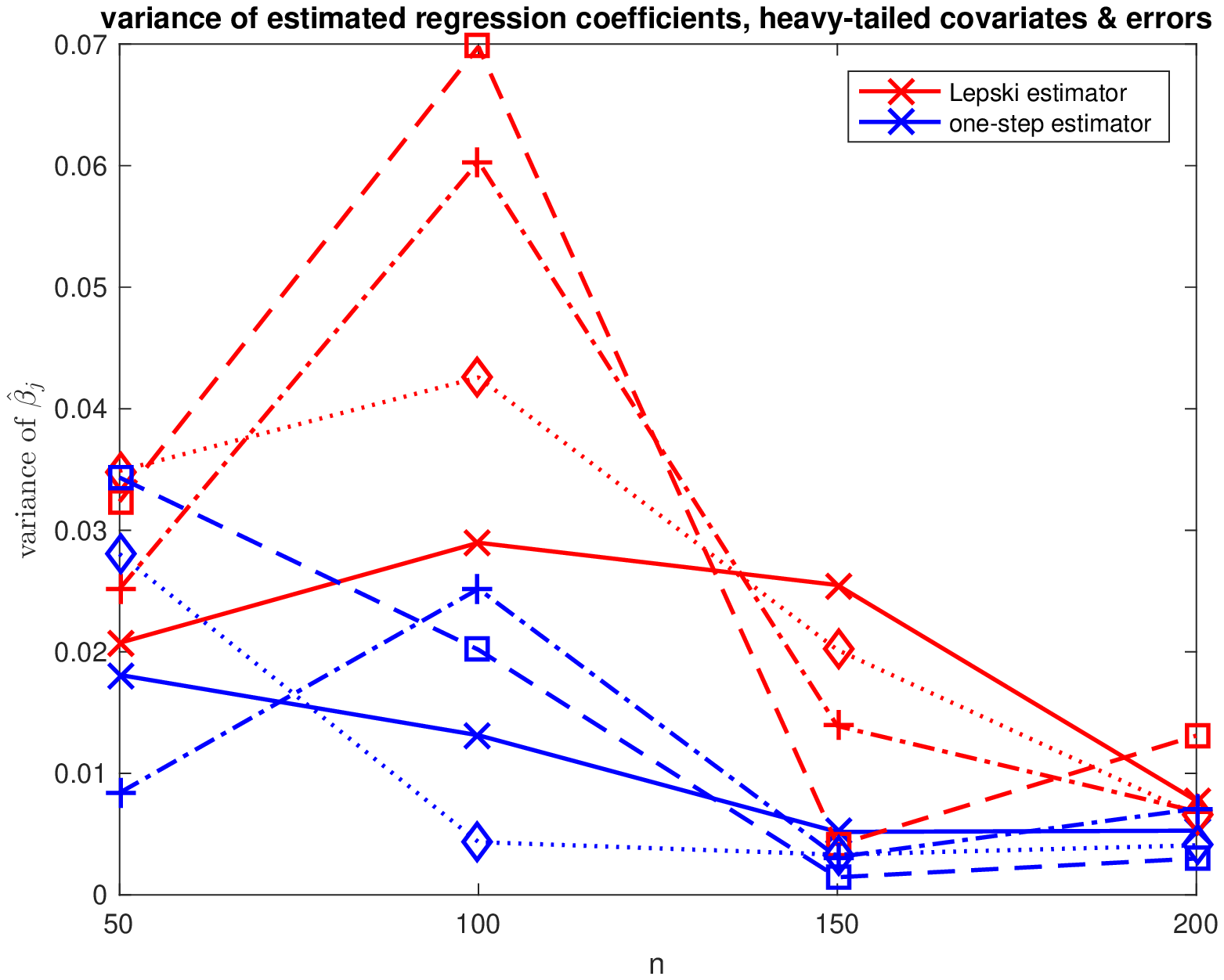} \\
(a) & (b)
\end{tabular}
\caption{Plots comparing $\ell_2$-error and variance of estimators obtained via Lepski's method (red) and Lepski's method followed by a one-step correction (blue), when $p = 100$ and $k = 4$. Both covariates and errors were generated from $t$-distributions with 3 degrees of freedom. Panel (a) shows the error $\|\betahat - \betastar\|_2$, averaged over 10 trials. Panel (b) shows the empirical variance of $\betahat_j$, for each of the four nonzero regression coefficients, computed with respect to 10 trials. The coefficients are distinguished in the figure using different line markings.}
\label{FigL2Var2}
\end{center}
\end{figure}


We also provide a set of simulation results illustrating the validity of our method for constructing confidence intervals. We simulated data from a linear model with $t$-distributed additive errors. We then constructed confidence intervals according to the method of Section~\ref{SecConf}. For comparison, we used the method of van de Geer et al.~\cite{vanEtal14} to construct confidence intervals as if the errors were Gaussian, beginning with the consistent estimator obtained using our adaptive Huber estimator.

\begin{figure}
\begin{center}
\begin{tabular}{cc}
\includegraphics[width=0.48\textwidth]{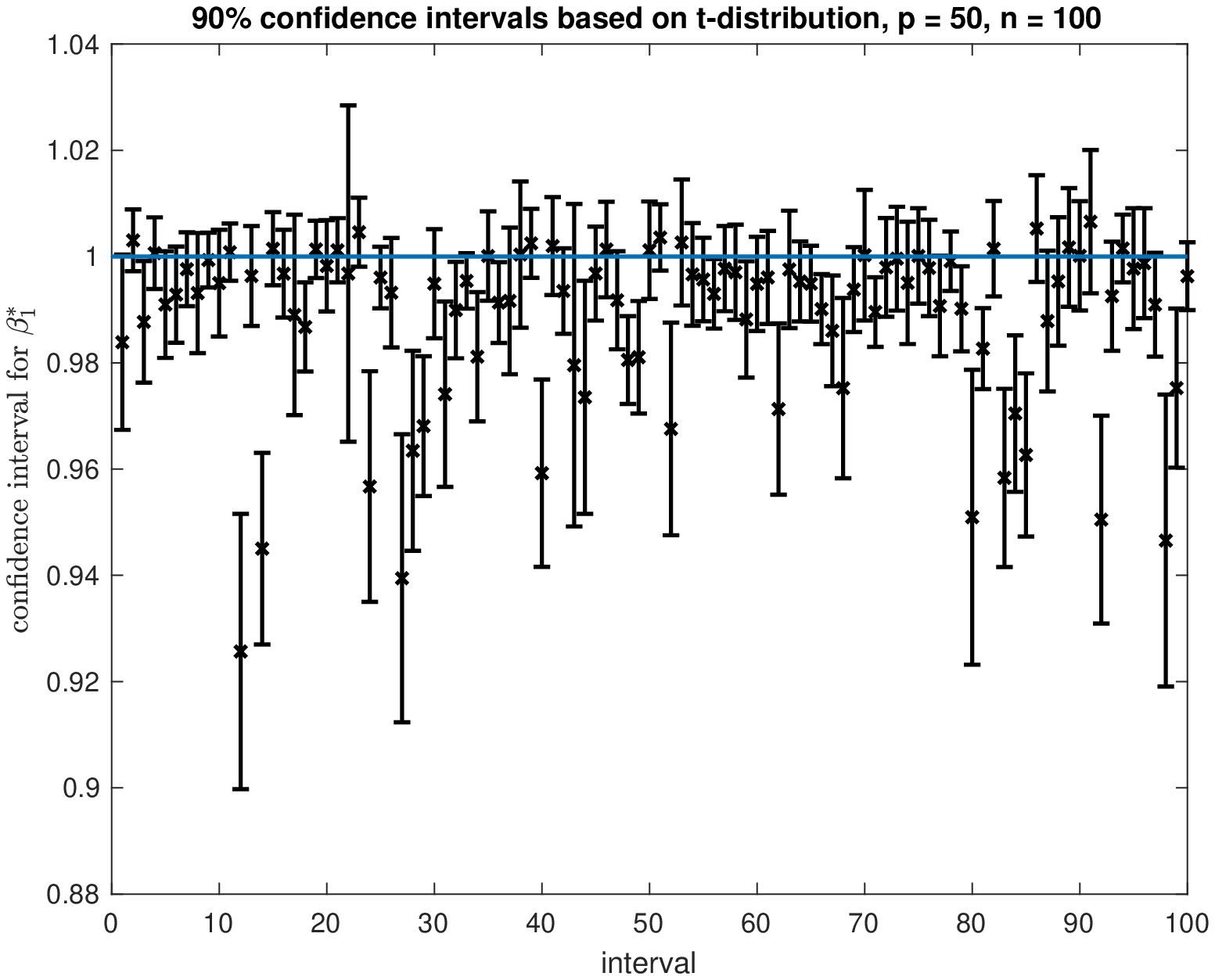} & \includegraphics[width=0.48\textwidth]{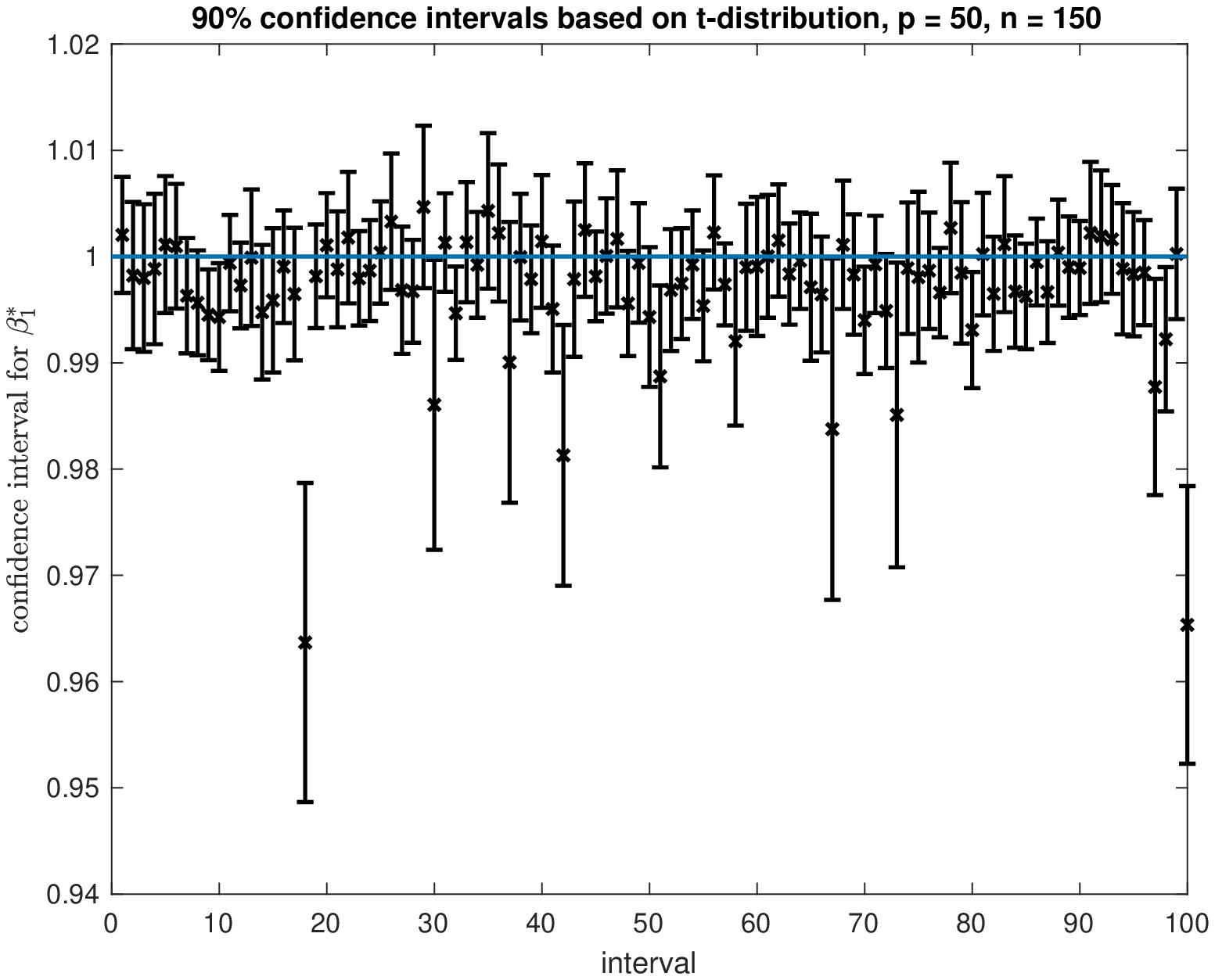} \\
(a) & (b)
\end{tabular}
\begin{tabular}{c}
\includegraphics[width=0.48\textwidth]{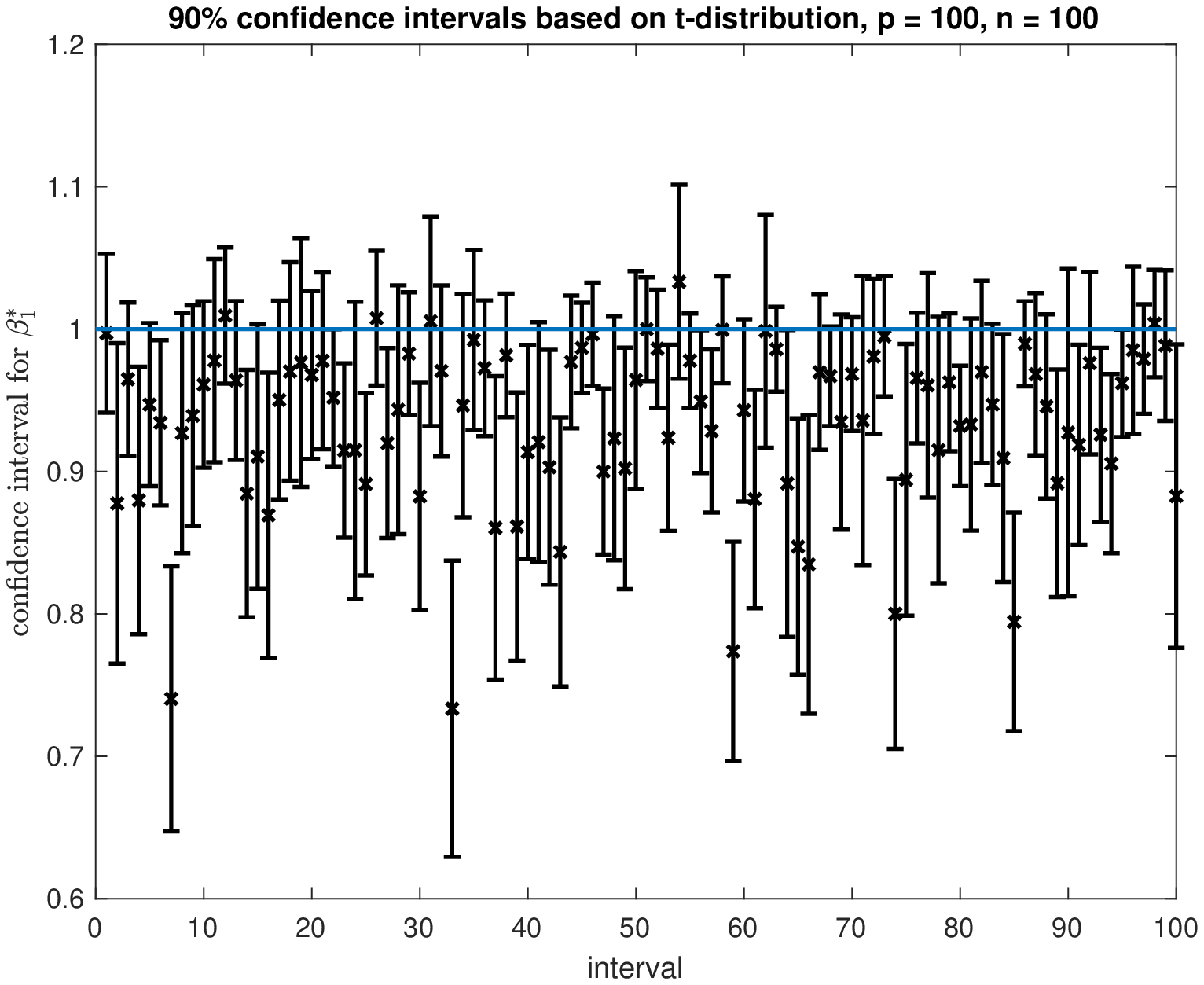} \\
(c)
\end{tabular}
\caption{Plots showing results of confidence interval simulations based on 100 trials. Data were generated with covariates and errors drawn from $t$-distributions with 3 degrees of freedom, and confidence intervals were constructed at the 90\% level. Panels (a) and (b) show confidence interval coverage for $p = 50$ with $n = 100$ and $n=150$. The empirical coverage was 67\% and 85\%, respectively. Panel (c) shows confidence interval coverage for $p = 100$ and $n = 100$. The empirical coverage was 62\%.}
\label{FigConfidence}
\end{center}
\end{figure}

We observe that the confidence intervals constructed according to our procedure have the proper empirical coverage, whereas the intervals constructed using a Gaussian error assumption are generally too narrow. In comparison to the $t$-interval coverage reported in Figure~\ref{FigConfidence}, the coverage levels for confidence intervals constructed using a normality assumption were (a) 96\%, (b) 100\%, and (c) 95\%. This illustrates the importance of using an approximately valid error distribution for valid inference. To check for consistency as $n \rightarrow \infty$, we ran the same confidence interval experiment with $p = 10$ and $n = 100$, and observed that the empirical coverage was 90\% for confidence intervals computed with respect to the correct $t$-distribution, compared to 88\% for confidence intervals computed with respect to the normal distribution. This agrees with the theoretical predictions that confidence intervals based on the normal assumption will generally be overly conservative, and the $t$-intervals indeed achieve the proper level of coverage when $n$ is sufficiently large relative to $p$.


\subsection{Real data experiment}

We also analyzed a data set collected from X-ray microanalysis of archaeological glass vessels~\cite{JanEtal98}. This data set has been analyzed in several other papers on high-dimensional robust linear regression with leverage points~\cite{Mar11, SmuYoh17}. The data set consists of $n = 180$ observations and $p = 486$ frequencies, which we use as predictors for the contents of compound 13, which is PbO. As discussed in \cite{Mar11}, the data set contains clear outliers.

Following the method of Smucler and Yohai~\cite{SmuYoh17} for tuning parameter selection, we chose the parameter $\lambda$ in our algorithm via 5-fold cross-validation using a $\tau$-scale of the residuals~\cite{YohZam88, SalEtal08}. (Note that our theorems are stated with $\lambda$ equal to $\sqrt{\frac{\log p}{n}}$ times universal constants, but in practice, choosing $\lambda$ in a data-driven manner leads to better predictive performance.) Based on this procedure, Lepski's method yielded a sparse vector with only one nonzero component, corresponding to frequency number 154. This fit corresponds to the value 0.126 of the $\tau$-scale, which is comparable to the values reported in Smucler and Yohai~\cite{SmuYoh17} using alternative methods: $MM$-Lasso (0.086), adaptive $MM$-Lasso (0.083), sparse-LTS (0.329), Lasso (0.131), and adaptive Lasso (0.138).

We also attempted to construct confidence intervals for the selected frequency. The simulations were inconclusive, due to the fact that various implementations of the graphical Lasso algorithm on the $486 \times 486$ matrix of covariates failed to converge. We suspect that this is because the assumption that the population-level inverse covariance matrix is sparse is violated, or the covariate distribution is heavy-tailed and/or possesses extreme outliers, so that the rate of convergence of the sample covariance matrix to its mean is too slow. This experiment reveals that the additional assumptions required to construct confidence intervals may be somewhat more stringent than the assumptions needed for consistency in terms of estimation or prediction error.


\section{Discussion}

Throughout this paper, we have assumed that second moments of the $\epsilon_i$'s and $x_i$'s exist. We now describe a small adaptation that applies in the case when second moments do not exist; it is still possible to obtain a consistent estimator, at the expense of efficiency. The two places where we have required existence of second moments in our analysis are (a) in the computation of the rough scale parameter bounds $\sigma_{\min}$ and $\sigma_{\max}$; and (b) in the matrix inversion step where we estimate the inverse of the covariance matrix $\Sigma_x$.

We begin by discussing item (a). We may use the median absolute deviation (MAD) as the scale parameter when the second moments are not finite. Recall that the population-level MAD is given by
\begin{equation*}
\MAD(X) = \med\left(|X - \med(X)|\right),
\end{equation*}
where $\med$ denotes the median operator. By Lemma~\ref{LemMADsum} in Appendix~\ref{AppUseful}, we know that under the assumption that the distribution of $\epsilon_i$ is symmetric and unimodal, we have
\begin{equation*}
\MAD(\epsilon_i) \le \MAD(x_i^T \betastar + \epsilon_i) = \MAD(y_i),
\end{equation*}
so that the MAD estimate based on the $y_i$'s can indeed be used as an upper bound on the scale of the $\epsilon_i$'s, analogous to the case of the variance.

For item (b), we may use a slight modification of the one-step estimator, such as the following (cf.\ Welsh and Ronchetti~\cite{WelRon02}):
\begin{itemize}
\item $\frac{1}{n} \sum_{i=1}^n x_i x_i^T w_i \frac{\psi_\sigma(v_i r_i)}{r_i}$
\item $\frac{1}{n} \sum_{i=1}^n \psi_\sigma'(r_i) \cdot \frac{1}{n} \sum_{i=1}^n w_i x_i x_i^T$.
\end{itemize}
Under appropriate assumptions on the tails and moments of the covariates, similar consistency and asymptotic normality results could be derived for estimators based on these quantities.

We also mention two interesting open questions that are not addressed by our theory. The first is what one might do in the case when the distribution of $\epsilon_i$ is not necessarily symmetric. The second is of more practical relevance: What type of one-step estimator could we use for obtaining a more efficient estimator and/or confidence intervals when the shape of the error distribution is unknown? Some general guidelines for choosing the $\psi$ function in the one-step estimator, or a more principled procedure for flagging outliers and then fitting confidence intervals based on a fitted distribution, would be quite useful in practice.

Finally, an interesting direction to pursue would be whether an approach based on Lepski's method could also be used to adaptively choose the correct parameter for the Huber loss in the case of an $\epsilon$-contaminated model (either in location estimation or linear regression). A related question is how to adaptively choose a trimming parameter for the robust location estimator based on trimmed means. These are both questions of theoretical interest that have largely remained open in the classical robust statistics literature---since they depend on minimizing variance quantities, rather than deriving high-probability error bounds, the machinery developed in this paper does not carry over directly. However, it is plausible that an appropriate modification of the Lepski-based approach may result in theoretically valid conclusions for obtaining a near-optimal estimator from the point of view of variance.



\section*{Acknowledgments}

The author would like to thank Ezequiel Smucler for sharing the archaeological dataset used in the simulations. Part of this work was completed while the author was visiting the Isaac Newton Institute in Cambridge, UK. The author is supported by NSF grant DMS-1749857.

\appendix

\section{Semiparametric efficiency}
\label{AppSemi}

In this Appendix, we review some concepts in semiparametric estimation. For a more detailed overview, we refer the reader to the textbooks by Bickel et al.~\cite{BicEtal93} or Hansen~\cite{Han17}.

Following the treatment of Newey~\cite{New90}, we first define the semiparametric regression model~\cite{EngEtal86}:
\begin{definition}
The \emph{semiparametric regression model} characterized by a parameter vector $\beta_0 \in \real^q$ and function $g_0$ is given by
\begin{equation}
\label{EqnSemiLin}
y_i = x_i^T \beta_0 + g_0(v_i) + \epsilon_i, \qquad \text{for } 1 \le i \le n,
\end{equation}
where the $x_i$'s and $v_i$'s are vectors of exogenous observations, $y_i$ is a scalar response, and $\epsilon_i$ is independent additive error.
\end{definition}

We assume the distribution of the $\epsilon_i$'s is unknown, and our goal is to estimate the unknown vector $\beta_0$ from i.i.d.\ observations $\{(y_i, x_i, v_i)\}_{i=1}^n$. Recall the notion of efficiency:

\begin{definition}
An estimate $\betahat$ of $\beta_0$ is \emph{semiparametrically efficient} if it is regular (i.e., $\sqrt{n} (\betahat - \beta_0)$ is asymptotically normal), and the asymptotic variance is minimal among all regular estimates of $\beta_0$.
\end{definition}

Semiparametric efficiency is usually established by obtaining lower bounds on the asymptotic variance of an efficient estimator by considering Cramer-Rao bounds for different parametric ``submodels," which are models that include the semiparametric model under consideration and are equal to the semiparametric model for a certain value of the parameter. In particular, the Cramer-Rao bound for any parametric subclass must provide a lower bound for the semiparametric estimation problem, as well, and we have the variance lower bound
\begin{equation*}
\widebar{V} = \sup_\theta V_\theta,
\end{equation*}
where $V_\theta$ is the Cramer-Rao bound corresponding to a parametric submodel indexed by $\theta$. If one can find a parametric submodel with a Cramer-Rao bound that matches the asymptotic variance of a particular semiparametric estimator, that estimator is guaranteed to be efficient. Note that for multidimensional problems, the supremum is taken with respect to the partial order of positive semidefinite matrices (and the supremum is guaranteed to exist under appropriate regularity conditions, which apply in the setting considered here).

Newey~\cite{New90} presents an approach to compute the variance bound $\widebar{V}$ directly by considering the projection of the score function of the semiparametric model onto the tangent set corresponding to the scores of all parametric submodels, where the score of the semiparametric model is the partial derivative of the negative log likelihood with respect to the parameter vector. To be more formal, consider a parametric submodel parametrized by $\theta = (\beta, \eta)$, where both $\beta$ and $\eta$ are vectors, and $\beta$ corresponds to the $q$-dimensional parametric part of the original semiparametric model. The overall score function may be partitioned as $S_\theta = (S_\beta, S_\eta)$. By block matrix inversion, we may verify that the Cramer-Rao bound for estimation of $\beta$ in the parametric submodel is then given by
\begin{equation*}
V_\theta = \left(\E[(S_\beta - \Btil S_\eta)(S_\beta - \Btil S_\eta)^T]\right)^{-1},
\end{equation*}
where $\Btil := \E[S_\beta S_\eta^T] \left(E[S_\eta S_\eta^T]\right)^{-1}$. In particular, $\Btil S_\eta$ is the best linear predictor of $S_\beta$ as a function of $S_\eta$.

We now define the tangent set to be the mean square closure of all $q$-dimensional linear combinations of scores of parametric submodels:
\begin{equation*}
\scriptT = \left\{\scriptS \in \real^q: \E[\|\scriptS\|_2^2] < \infty, \; \exists A_j S_{\theta_j} \text{ s.t. } \E[\|\scriptS -  A_j S_{\theta_j}\|_2^2] \right\},
\end{equation*}
where the $A_j$'s are matrices with $q$ rows and the $S_{\theta_j}$'s are the score vectors of various parametric submodels.

We have the following result, which holds generally for semiparametric estimation (not just in the case of the semiparametric regression model):
\begin{lem*} [Theorem 3.2 of Newey~\cite{New90}]
\label{LemNewVar}
Suppose $\scriptT$ is a linear space, and let $S^\scriptT_\beta$ denote the projection of $S_\beta$ on $\scriptT$. Then
\begin{equation*}
\widebar{V} = \left(\E\left[(S_\beta - S^\scriptT_\beta) (S_\beta - S^\scriptT_\beta)^T\right]\right)^{-1},
\end{equation*}
provided the matrix is nonsingular.
\end{lem*}

For the model~\eqref{EqnSemiLin}, we denote a parametrization of $g_0(v)$ as $g(v,\eta)$, where $\eta$ is a parameter such that $g(v, \eta_0) = g_0(v)$. Then the log likelihood may be written as
\begin{equation*}
p_{\beta, \eta} (y|x, v) = \log f\left(y - x^T \beta + g(v, \eta)\right),
\end{equation*}
where $f$ is the density of $\epsilon_i$. Taking partial derivatives and evaluating at the true parameter values $(\beta_0, \eta_0)$, we obtain the score functions
\begin{equation*}
S_\beta = \frac{f'(\epsilon)}{f(\epsilon)} \cdot x, \qquad S_\eta = \frac{f'(\epsilon)}{f(\epsilon)} \cdot g_\eta,
\end{equation*}
where $\epsilon = y - x^T \beta_0 - g_0(v)$ and $g_\eta \defn \frac{\partial g(v, \eta)}{\partial \eta} \Big |_{\eta = \eta_0}$. It is not hard to verify that the tangent set is equal to
\begin{equation*}
\scriptT = \left\{\frac{f'(\epsilon)}{f(\epsilon)} \cdot D(v): \; \E\left[\left(\frac{f'(\epsilon)}{f(\epsilon)}\right)^2 \|D(v)\|^2\right] < \infty \right\},
\end{equation*}
using the observation that the parametric submodel with $g(v, \eta) = g_0(v) + \eta^T D(v)$ yields the score $S_\eta = \frac{f'(\epsilon)}{f(\epsilon)} \cdot D(v)$. Furthermore, $\scriptT$ is clearly a linear space.

In order to compute $S^\scriptT_\beta$, we use the following result:
\begin{lem*} [Lemma 3.4 in Newey~\cite{New90}]
\label{LemNewProj}
If $UW$ has finite second moment and $V$ and $W$ are functions of some random variable $T$, such that $\E[UU^T \mid T]$ is constant and positive definite, then the projection of $UW$ on the space
\begin{equation*}
\scriptT_V := \left\{U D(V): \; \E[\|U D(V)\|_2^2] < \infty \right\}
\end{equation*}
is equal to $U \E[W \mid V]$.
\end{lem*}
Applying Lemma~\ref{LemNewProj} with $W = x$, $V = v$, and $U = \frac{f'(\epsilon)}{f(\epsilon)}$, we conclude that
\begin{equation*}
S^\scriptT_\beta = \frac{f'(\epsilon)}{f(\epsilon)} \cdot \E[x|v].
\end{equation*}
Combining this with Lemma~\ref{LemNewVar}, we arrive at the following result:
\begin{thm*}
\label{ThmSemiEff}
Suppose $x$ has finite second moments and
\begin{equation*}
0 < \E\left[\left(\frac{f'(\epsilon)}{f(\epsilon)}\right)^2\right] < \infty.
\end{equation*}
Then
\begin{equation*}
\widebar{V} = \left(\E\left[\left(\frac{f'(\epsilon)}{f(\epsilon)}\right)^2 \cdot (x - \E[x|v])(x-\E[x|v])^T\right]\right)^{-1},
\end{equation*}
provided the matrix is nonsingular.
\end{thm*}


\section{Proof of Theorem~\ref{ThmHubOutput}}
\label{AppThmHubOutput}

The success of our proposed method depends critically on Theorem~\ref{ThmHubOutput}. We now provide the analysis of the bound supplied in the theorem.

We begin by analyzing the estimator
\begin{equation*}
\betatil_{\tau} \in \arg\min_{\|\beta - \betastar\|_2 \le \frac{\tau}{2b'}} \left\{\frac{1}{n} \sum_{i=1}^n \ell_\tau\left((x_i^T \beta - y_i)w(x_i)\right) w(x_i) + \lambda \tau \|\beta\|_1\right\},
\end{equation*}
where we have introduced a side constraint, depending on a parameter $b'$ to be specified later. We will show that such optima $\betatil_\tau$ lie in the interior of the constraint set, hence agree with the global optima $\betahat_\tau$ of the unconstrained problem.


\subsection{Basic inequality}
\label{AppBasic}

Let
\begin{equation*}
\Loss_n(\beta) := \frac{1}{n} \sum_{i=1}^n \ell_\tau\left((x_i^T \beta - y_i)w(x_i)\right) w(x_i).
\end{equation*}
We will derive an $\ell_2$-error bound on $\|\betatil_\tau - \betastar\|_2$, assuming the following conditions:
\begin{itemize}
\item (Regularization parameter)
\begin{equation}
\label{EqnReg}
\|\nabla \Loss_n(\betastar)\|_\infty \le \frac{\lambda \tau}{2}
\end{equation}
\item (RSC condition)
\begin{multline}
\label{EqnRSC}
\Loss_n(\beta) - \Loss_n(\betastar) - \inprod{\nabla \Loss_n(\betastar)}{\beta - \betastar} \ge \alpha \|\beta - \betastar\|_2^2, \\
\forall \beta \text{ s.t. } \|\Delta\|_2 \le \frac{\tau}{2b_0} \text{ and } \|\Delta_{S^c}\|_1 \le 3 \|\Delta_S\|_1,
\end{multline}
where we have denoted $\Delta := \beta - \betastar$.
\end{itemize}
In Appendices~\ref{AppReg} and~\ref{AppRSC}, we will show that the conditions~\eqref{EqnReg} and~\eqref{EqnRSC} hold w.h.p.

We have the basic inequality
\begin{equation}
\label{EqnBasic1}
\Loss_n(\betatil_{\tau}) + \lambda \tau \|\betatil_{\tau}\|_1 \le \Loss_n(\betastar) + \lambda \tau \|\betastar\|_1.
\end{equation}
Hence,
\begin{equation}
\label{EqnBasic2}
\inprod{\nabla \Loss_n(\betastar)}{\betatil_{\tau} - \betastar} \le \Loss_n(\betatil_{\tau}) - \Loss_n(\betastar) \le \lambda \tau \left(\|\betastar\|_1 - \|\betatil_\tau\|_1\right),
\end{equation}
where the first inequality is due to the convexity of $\Loss_n$. Therefore, we have
\begin{equation*}
0 \le \lambda \tau \left(\|\betastar\|_1 - \|\betatil_{\tau}\|_1\right) + \|\nabla \Loss_n(\betastar)\|_\infty \|\betatil_{\tau} - \betastar\|_1.
\end{equation*}
Denoting $\nuhat = \betatil_{\tau} - \betastar$ and using the bound~\eqref{EqnReg}, we then have
\begin{equation*}
0 \le \lambda \tau \left(\|\nuhat_S\|_1 - \|\nuhat_{S^c}\|_1\ + \frac{1}{2} \|\nuhat\|_1\right),
\end{equation*}
since
\begin{equation*}
\|\betastar\|_1 - \|\betatil_\tau\|_1 = \|\betastar_S\|_1 - \|\betatil_{\tau, S}\|_1 - \|\betatil_{\tau, S^c}\|_1 \le \|\nuhat_S\|_1 - \|\nuhat_{S^c}\|_1.
\end{equation*}
This implies that
\begin{equation*}
\|\nuhat_{S^c}\|_1 \le 3 \|\nuhat_S\|_1,
\end{equation*}
which is the cone condition.

Therefore, the RSC condition together with the basic inequality~\eqref{EqnBasic1} implies that
\begin{equation*}
\inprod{\nabla \Loss_n(\betastar)}{\nuhat} + \alpha \|\nuhat\|_2^2 \le \Loss_n(\betahat_{\tau}) - \Loss_n(\betastar) \le \lambda \tau \left(\|\betastar\|_1 - \|\betahat_{\tau}\|_1\right),
\end{equation*}
so combining with inequality~\eqref{EqnBasic2} and the assumptions, we have
\begin{equation*}
\alpha \|\nuhat\|_2^2 \le \lambda \tau \left(\|\nuhat_S\|_1 - \|\nuhat_{S^c}\|_1 + \frac{1}{2} \|\nuhat\|_1\right) \le \frac{3 \lambda \tau}{2} \|\nuhat_S\|_1 \le \frac{3 \lambda \tau \sqrt{k}}{2} \|\nuhat\|_2,
\end{equation*}
implying that
\begin{equation*}
\|\nuhat\|_2 \le \frac{3\lambda \tau \sqrt{k}}{2\alpha},
\end{equation*}
as claimed.


\subsection{Bound on regularization parameter}
\label{AppReg}

We now verify the bound~\eqref{EqnReg}. Note that
\begin{equation*}
\nabla \Loss_n(\betastar) = \frac{1}{n} \sum_{i=1}^n \ell'_\tau\left(\epsilon_i w(x_i)\right) w^2(x_i) x_i.
\end{equation*}
For each $1 \le j \le p$, we have
\begin{equation*}
\left|e_j^T \cdot \ell_\tau'(\epsilon_i w(x_i)) w^2(x_i) x_i\right| \le \tau \|w(x_i) x_i\|_2 \cdot |w(x_i)| \le \tau b'.
\end{equation*}
Furthermore, note that
\begin{equation*}
\E[\ell'_\tau(\epsilon_i w(x_i)) w^2(x_i) x_i] = 0,
\end{equation*}
since
\begin{equation*}
\E[\ell'_\tau(\epsilon_i w(x_i)) \mid x_i] = 0,
\end{equation*}
using the fact that $\ell_\tau$ is an even function and $\epsilon_i$ is independent of $x_i$ and has a symmetric distribution. Hence, a simple application of Hoeffding's inequality, together with a union bound, yields
\begin{equation*}
\|\nabla \Loss_n(\betastar)\|_\infty \le c \tau b_0 \sqrt{\frac{\log p}{n}},
\end{equation*}
with probability at least $1 - c\exp(-c'n)$, provided $n \succsim \log p$. In particular, the choice of regularization parameter $\lambda = 2c b_0 \sqrt{\frac{\log p}{n}}$ ensures that $\|\nabla \Loss_n(\betastar)\|_\infty \le \frac{\lambda \tau}{2}$, w.h.p.

\subsection{RSC condition}
\label{AppRSC}

We now turn to the more challenging task of establishing the RSC condition~\eqref{EqnRSC}. We show that w.h.p., the inequality
\begin{equation*}
\Loss_n(\beta) - \Loss_n(\betastar) - \inprod{\nabla \Loss_n(\betastar)}{\beta - \betastar} \ge \alpha \|\beta - \betastar\|_2^2
\end{equation*}
holds uniformly over the set
\begin{equation*}
\mathbb{C} := \Big\{\beta: \; \|\beta - \betastar\|_2 \le \frac{\tau}{2b_0}, \; \|\beta_{S^c} - \betastar_{S^c}\|_1 \le 3 \|\beta_S - \betastar_S\|_1\Big\}.
\end{equation*}
Defining $\scriptT(\beta, \betastar) := \Loss_n(\beta) - \Loss_n(\betastar) - \inprod{\nabla \Loss_n(\betastar)}{\beta - \betastar}$, we have
\begin{multline*}
\scriptT(\beta, \betastar) = \frac{1}{n} \sum_{i=1}^n \Big\{\ell_\tau\left((x_i^T \beta - y_i)w(x_i)\right) - \ell_\tau\left(\epsilon_i w(x_i)\right) - \ell_\tau'\left(\epsilon_i w(x_i)\right) w(x_i) x_i^T (\beta - \betastar)\Big\} w(x_i).
\end{multline*}
Further note that for $|u_1|, |u_2| \le \tau$, we have
\begin{equation*}
\ell_\tau(u_1) - \ell_\tau(u_2) - \ell'_\tau(u_2) (u_1 - u_2) = \frac{(u_1 - u_2)^2}{2},
\end{equation*}
whereas the convexity of $\ell_\tau$ implies that
\begin{equation*}
\ell_\tau(u_1) - \ell_\tau(u_2) - \ell'_\tau(u_2) (u_1 - u_2) \ge 0, \qquad \forall u_1, u_2 \in \real.
\end{equation*}
Define the events
\begin{equation*}
A_i := \left\{|\epsilon_i| \le \frac{\tau}{2}\right\}, \qquad \forall 1 \le i \le n,
\end{equation*}
and note that when $A_i$ holds, we have
\begin{equation*}
|\epsilon_i w(x_i)| \le |\epsilon_i| \le \frac{\tau}{2},
\end{equation*}
so
\begin{align*}
|(x_i^T \beta - y_i) w(x_i)| & \le |w(x_i) x_i^T(\beta - \betastar)| + |\epsilon_i w(x_i)| \\
& \le \|w(x_i) x_i\|_2 \|\beta - \betastar\|_2 + \frac{\tau}{2} \\
& \le b_0 \cdot \frac{\tau}{2b_0} + \frac{\tau}{2} \\
& \le \tau.
\end{align*}
Thus,
\begin{equation*}
\scriptT(\beta, \betastar) \ge \frac{1}{n} \sum_{i=1}^n \frac{1}{2} \left(w(x_i) x_i^T (\beta - \betastar)\right)^2 w(x_i) 1_{A_i}.
\end{equation*}

We make use of the following result:
\begin{lem*} 
\label{LemConeDev}
[Lemma 12 in Loh and Wainwright~\cite{LohWai11a}]
For a fixed matrix $\Gamma \in \real^{p \times p}$, parameter $s \ge 1$, and tolerance $\delta > 0$, suppose we have the deviation condition
\begin{equation}
\label{EqnDevAssump}
|v^T \Gamma v| \le \delta, \qquad \forall v \in \real^p \text{ s.t. } \|v\|_0 \le 2s, \quad \|v\|_2 \le 1.
\end{equation}
Then
\begin{equation*}
|v^T \Gamma v| \le 27 \delta \left(\|v\|_2^2 + \frac{\|v\|_1^2}{s}\right), \qquad \forall v \in \real^p.
\end{equation*}
\end{lem*}

We will apply the lemma to the matrix
\begin{equation*}
\Gamma = \frac{1}{n} \sum_{i=1}^n \frac{w^3(x_i)}{2} 1_{A_i} \cdot x_i x_i^T - \E\left[\frac{w^3(x_i)}{2} 1_{A_i} \cdot x_i x_i^T\right],
\end{equation*}
with $s = k$ and $\delta = C_\delta \lambda_{\min}\left(\E\left[\frac{w^3(x_i)}{2} x_i x_i^T\right]\right)$, for a value of $C_\delta$ specified below. (We will verify the deviation condition~\eqref{EqnDevAssump} momentarily.)

Denoting $\Delta := \beta - \betastar$, we then have
\begin{multline*}
\frac{1}{n} \sum_{i=1}^n \frac{1}{2} \left(w(x_i) x_i^T \Delta \right)^2 w(x_i) 1_{A_i} \ge \E\left[\frac{w^3(x_i)}{2} 1_{A_i} \left(x_i^T \Delta \right)^2\right] - 27\delta \left(\|\Delta\|_2^2 + \frac{\|\Delta\|_1^2}{k}\right),
\end{multline*}
uniformly over all $\Delta \in \real^p$. Now note that for any $\Delta$, we have
\begin{align*}
\E\left[\frac{w^3(x_i)}{2} 1_{A_i} \left(x_i^T \Delta \right)^2\right] & = \E\left[\frac{w^3(x_i)}{2} (x_i^T \Delta)^2\right] \cdot \mprob\left(|\epsilon_i| \le \frac{\tau}{2}\right) \\
& \ge \lambda_{\min}\left(\E\left[\frac{w^3(x_i)}{2} x_i x_i^T\right]\right) \|\Delta\|_2^2 \cdot \left(1 - \frac{(\sigmastar)^2}{\tau^2/4}\right) \\
& \ge \frac{5}{9}  \lambda_{\min}\left(\E\left[\frac{w^3(x_i)}{2} x_i x_i^T\right]\right) \|\Delta\|_2^2,
\end{align*}
where we have used Chebyshev's inequality and the fact that $\sigmastar \le \frac{\tau}{3}$. Furthermore, for $\Delta \in \mathbb{C}$, we have
\begin{equation*}
\|\Delta\|_1 = \|\Delta_S\|_1 + \|\Delta_{S^c}\|_1 \le 4 \|\Delta_S\|_1 \le 4 \sqrt{k} \|\Delta_S\|_2,
\end{equation*}
so
\begin{equation*}
\|\Delta\|_2^2 + \frac{\|\Delta\|_1^2}{k} \le 17 \|\Delta\|_2^2.
\end{equation*}
It follows that
\begin{align*}
\frac{1}{n} \sum_{i=1}^n \frac{1}{2} \left(w(x_i) x_i^T \Delta \right)^2 w(x_i) 1_{A_i} & \ge \frac{5}{9} \lambda_{\min}\left(\E\left[\frac{w^3(x_i)}{2} x_i x_i^T\right]\right) \|\Delta\|_2^2  \\
& \qquad \qquad \qquad - 459\delta \|\Delta\|_2^2 \\
& \ge \frac{5}{18} \lambda_{\min}\left(\E\left[\frac{w^3(x_i)}{2} x_i x_i^T\right]\right) \|\Delta\|_2^2,
\end{align*}
for the choice $C_\delta = \frac{5}{18 \cdot 459}$.

Finally, note that the bound~\eqref{EqnDevAssump} in the hypothesis of Lemma~\ref{LemConeDev} holds,
w.h.p. Indeed, for $\|v\|_2 \le 1$, the quantity $v^T \Gamma v$ is the recentered average of i.i.d.\ bounded random variables, where
\begin{equation*}
\left|\frac{w^3(x_i)}{2} 1_{A_i} \cdot (x_i^T v)^2\right| \le \frac{1}{2} \cdot \|w(x_i) x_i\|_2^2 \|v\|_2^2 \le (b')^2,
\end{equation*}
so Hoeffding's inequality, along with an $\epsilon$-net argument over $2k$-dimensional subspaces and union bound over the $\binom{p}{2k}$ choices of the support set, implies that
\begin{multline}
\label{EqnCovering}
\mprob\left(|v^T \Gamma v| \le \delta, \quad \forall v \in \real^p \text{ s.t. } \|v\|_0 \le 2k, \; \|v\|_2 \le 1\right) \ge 1 - 2\exp\left(-\frac{c\delta n}{(b')^2} + 2k \log p\right)
\end{multline}
(cf.\ Lemma 15 in Loh and Wainwright~\cite{LohWai11a}). Hence, provided $n \succsim \frac{(b')^2}{\lambda_{\min}\left(\E\left[\frac{w^3(x_i)}{2} x_i x_i^T\right]\right)} k \log p$, we have the desired uniform deviation bound with probability at least $1 - 2\exp(-c'n)$. Combining the results, we obtain the desired RSC condition~\eqref{EqnRSC} with $\alpha = \frac{5}{18} \lambda_{\min}\left(\E\left[\frac{w^3(x_i)}{2} x_i x_i^T\right]\right)$, with probability at least $1 - 2\exp(-c'n)$.

\subsection{Conclusion of proof}

Altogether, we conclude that for $\tau \ge 3 \sigmastar$, we have
\begin{equation*}
\|\betatil_{\tau} - \betastar\|_2 \le \frac{Cb' \tau}{\lambda_{\min}\left(\E\left[\frac{w^3(x_i)}{2} x_i x_i^T\right]\right)} \sqrt{\frac{k\log p}{n}},
\end{equation*}
with probability at least $1 - c \exp(-c'n)$. Further note that for $n \succsim k \log p$, we are guaranteed that
\begin{equation*}
\frac{Cb' \tau}{\lambda_{\min}\left(\E\left[\frac{w^3(x_i)}{2} x_i x_i^T\right]\right)} \sqrt{\frac{k\log p}{n}} < \frac{\tau}{2b'}.
\end{equation*}
It follows that $\betatil_\tau$ is in the interior of the region $\left\{\beta: \; \|\beta - \betastar\|_2 \le \frac{\tau}{2b'}\right\}$, so $\betatil_\tau$ must also be a global optimum of the regularized Huber estimator~\eqref{EqnHuberReg} that does not include the side constraint; furthermore, any optima of the unconstrained problem must also lie in the interior of the constraint set.

This concludes the proof of the theorem.


\section{Proofs of additional theorems}

In this appendix, we provide the proofs of the main results in the paper, with proofs of supporting lemmas supplied in later appendices.

\subsection{Proof of Theorem~\ref{ThmLepski}}
\label{AppThmLepski}

Let $j' = \min\left\{j \in \scriptJ: \sigma_j \ge \sigmastar\right\}$. Then $\sigma_{j'} \le 2\sigmastar$. We have
\begin{align*}
\mprob(j_* > j') & = \mprob\Bigg(\bigcup_{i \in \scriptJ: i > j'} \left\{\|\betahat_{(i)} - \betahat_{(j')}\|_2 > 6C\sigma_i \sqrt{\frac{k \log p}{n}}\right\} \\
& \qquad \qquad \bigcup_{i \in \scriptJ: i > j'} \left\{\|\betahat_{(i)} - \betahat_{(j')}\|_1 > 24 C \sigma_i k \sqrt{\frac{\log p}{n}}\right\}\Bigg) \\
& \le \mprob\Bigg(\|\betahat_{(j')} - \betastar\|_2 > 3C\sigma_{j'} \sqrt{\frac{k \log p}{n}}, \\
& \qquad \qquad \|\betahat_{(j')} - \betastar\|_1 > 12 C \sigma_{j'} k \sqrt{\frac{\log p}{n}}\Bigg) \\
& \qquad + \sum_{i \in \scriptJ: i > j'} \mprob\Bigg(\|\betahat_{(i)} - \betastar\|_2 > 3C \sigma_i \sqrt{\frac{k \log p}{n}}, \\
& \qquad \qquad \qquad \qquad \|\betahat_{(i)} - \betastar\|_1 > 12 C \sigma_i k \sqrt{\frac{\log p}{n}}\Bigg) \\
& \le c \exp(-c'n) + \log_2\left(\frac{2\sigma_{\max}}{\sigma_{\min}}\right) \cdot c\exp(-c'n),
\end{align*}
where we have used Theorem~\ref{ThmHubOutput} and a union bound in the final inequality.

Hence, with probability at least $1-\log_2\left(\frac{4\sigma_{\max}}{\sigma_{\min}}\right) \cdot c\exp(-c'n)$, we have $j' \ge j_*$ and the bounds
\begin{align*}
\|\betahat_{(j')} - \betastar\|_2 & \le 3C\sigma_{j'} \sqrt{\frac{k \log p}{n}}, \\
\|\betahat_{(j')} - \betastar\|_1 & \le 12C \sigma_{j'} k \sqrt{\frac{\log p}{n}}.
\end{align*}
It follows that
\begin{align*}
\|\betahat_{(j_*)} - \betastar\|_2 & \le \|\betahat_{(j_*)} - \betahat_{(j')}\|_2 + \|\betahat_{(j')} - \betastar\|_2 \\
& \le 6C\sigma_{j'} \sqrt{\frac{k \log p}{n}} + 3C\sigma_{j'} \sqrt{\frac{k \log p}{n}} \\
& \le 9C \sigma_{j'} \sqrt{\frac{k \log p}{n}} \\
& \le 18C \sigma^* \sqrt{\frac{k \log p}{n}},
\end{align*}
using the fact that $\sigma_{j'} \le 2\sigmastar$ in the final inequality.

Similarly, we have the bound
\begin{equation*}
\|\betahat_{j_+} - \betastar\|_1 \le 72C \sigma^* k \sqrt{\frac{\log p}{n}}.
\end{equation*}


\subsection{Proof of Theorem~\ref{ThmAsympNorm}}
\label{AppThmAsympNorm}

We write
\begin{align}
\label{EqnPlum}
\sqrt{n}(\bhat_\psi - \betastar) & = \sqrt{n} (\betahat - \betastar) + \frac{\Thetahat}{\Ahat(\psi)} \cdot \frac{1}{\sqrt{n}} \sum_{i=1}^n \psi \left(\frac{y_i - x_i^T \betahat}{\sigmahat}\right) x_i \notag \\
& = \frac{\Thetahat}{\Ahat(\psi)} \cdot \frac{1}{\sqrt{n}} \sum_{i=1}^n \psi \left(\frac{\epsilon_i}{\sigmastar}\right) x_i + \sqrt{n} \Bigg\{(\betahat - \betastar) \notag \\
& \quad + \frac{\Thetahat}{\Ahat(\psi)} \cdot \frac{1}{n} \sum_{i=1}^n \left(\psi \left(\frac{y_i - x_i^T \betahat}{\sigmahat}\right) - \psi\left(\frac{y_i - x_i^T \betastar}{\sigmastar}\right)\right) x_i\Bigg\} \notag \\
& \qquad := I + II.
\end{align}

We first consider the term $I = \frac{\Thetahat}{\Ahat(\psi)} \cdot \frac{1}{\sqrt{n}} \sum_{i=1}^n \psi\left(\frac{\epsilon_i}{\sigmastar}\right) x_i$, which we claim is asymptotically normal. We have
\begin{align*}
& \left\|P_J\left(\frac{\Thetahat}{\Ahat(\psi)} - \frac{\Theta_x}{A(\psi)} \right) \frac{1}{\sqrt{n}} \sum_{i=1}^n \psi\left(\frac{\epsilon_i}{\sigmastar}\right) x_i\right\|_\infty \\
& \qquad \le \left\|\left(\frac{\Thetahat}{\Ahat(\psi)} - \frac{\Theta_x}{A(\psi)} \right) \frac{1}{\sqrt{n}} \sum_{i=1}^n \psi\left(\frac{\epsilon_i}{\sigmastar}\right) x_i\right\|_\infty \\
& \qquad \le \opnorm{\frac{\Thetahat}{\Ahat(\psi)} - \frac{\Theta_x}{A(\psi)}}_1 \cdot \left\|\frac{1}{\sqrt{n}} \sum_{i=1}^n \psi\left(\frac{\epsilon_i}{\sigmastar}\right) x_i\right\|_\infty.
\end{align*}
The second factor is asymptotically $\order(\sqrt{\log p})$, by appealing to Lemma~\ref{LemConcX}, which approximates the term as a supremum of $p$ Gaussian random variables. To handle the first factor, we write
\begin{align}
\label{EqnThetaRatio}
\opnorm{\frac{\Thetahat}{\Ahat(\psi)} - \frac{\Theta_x}{A(\psi)}}_1 & \le \opnorm{\frac{1}{A(\psi)} \left(\Thetahat - \Theta_x\right)}_1 + \opnorm{\left(\frac{1}{\Ahat(\psi)} - \frac{1}{A(\psi)}\right) \Theta_x}_1 \notag \\
& \qquad \qquad \qquad \qquad + \opnorm{\left(\frac{1}{\Ahat(\psi)} - \frac{1}{A(\psi)}\right)\left(\Thetahat - \Theta_x\right)}_1 \notag \\
& \le \frac{1}{|A(\psi)|} \opnorm{\Thetahat - \Theta_x}_1 + \left|\frac{1}{\Ahat(\psi)} - \frac{1}{A(\psi)}\right| \opnorm{\Theta_x}_1 \notag \\
& \qquad \qquad \qquad \qquad \qquad + \left|\frac{1}{\Ahat(\psi)} - \frac{1}{A(\psi)}\right| \opnorm{\Thetahat - \Theta_x}_1 \notag \\
& = \order\left(\frac{k \log p}{\sqrt{n}}\right),
\end{align}
where the final inequality leverages the conditions~\eqref{EqnACond} and~\eqref{EqnThetaCond}.

Together with the convergence statement~\eqref{EqnConvDist}, we conclude that $I$ has the desired asymptotic normality property, since $\order\left(\frac{k \log p}{\sqrt{n}}\right) \cdot \order(\sqrt{\log p}) = o_\mprob(1)$, assuming the scaling $n \succsim k^2 \log^3 p$.

We now shift our attention to term $II$ on the right-hand side of equation~\eqref{EqnPlum}. By Taylor's theorem applied to each summand, we have
\begin{multline*}
\frac{1}{n} \sum_{i=1}^n \left(\psi\left(\frac{y_i - x_i^T \betahat}{\sigmahat}\right) - \psi\left(\frac{y_i - x_i^T \betastar}{\sigmastar}\right)\right) x_i \\
=  - \frac{1}{n} \sum_{i=1}^n \psi'\left(\frac{\epsilon_i}{\sigmastar}\right) \deltahat_i x_i + \frac{1}{2n} \sum_{i=1}^n \psi''(\widehat{t}_i) \deltahat_i^2 x_i,
\end{multline*}
where $\deltahat_i := \frac{y_i - x_i^T \betahat}{\sigmahat} - \frac{y_i - x_i^T \betastar}{\sigmastar}$ and $\widehat{t}_i$ lies on the segment between $\frac{y_i - x_i^T \betastar}{\sigmastar}$ and $\frac{y_i - x_i^T \betahat}{\sigmahat}$. We have the bound
\begin{equation*}
\sqrt{n} \left\|\frac{\Thetahat}{\Ahat(\psi)} \cdot \frac{1}{n} \sum_{i=1}^n \frac{\psi''(\widehat{t}_i)}{2} \deltahat_i^2 x_i\right\|_\infty \le \frac{\sqrt{n}}{2} \opnorm{\frac{\Thetahat}{\Ahat(\psi)}}_1 \cdot \left\|\frac{1}{n} \sum_{i=1}^n \psi''(\widehat{t}_i) \deltahat_i^2 x_i\right\|_\infty,
\end{equation*}
and
\begin{align*}
\left\|\frac{1}{n} \sum_{i=1}^n \psi''(\widehat{t}_i) \deltahat_i^2 x_i\right\|_\infty & \le \|\psi''\|_\infty \left| \frac{1}{n} \sum_{i=1}^n \deltahat_i^2\right| \cdot \|X\|_{\max} \\
& = \order\left(\frac{k \log p}{n} \cdot \sqrt{\log p}\right),
\end{align*}
using the same argument employed to bound the term $B_3$ in the proof of Lemma~\ref{LemApsi} and the bound on $\|X\|_{\max}$ from Lemma~\ref{LemXmax}.
%
%
Altogether, we have the bound
\begin{equation*}
\sqrt{n} \left\|\frac{\Thetahat}{\Ahat(\psi)} \cdot \frac{1}{n} \sum_{i=1}^n \frac{\psi''(\widehat{t}_i)}{2} \deltahat_i^2 x_i\right\|_\infty = \order\left(k \sqrt{\frac{\log^3 p}{n}}\right) = \order_\mprob(1).
\end{equation*}

%
%
%

Finally, note that using the expansion $\deltahat_i = \frac{x_i(\betastar - \betahat)}{\sigmahat} + \epsilon_i \left(\frac{1}{\sigmahat} - \frac{1}{\sigmastar}\right)$, we have
\begin{align}
\label{EqnTomato}
& \left\|(\betahat - \betastar) - \frac{\Thetahat}{\Ahat(\psi)} \frac{1}{n} \sum_{i=1}^n \psi'\left(\frac{\epsilon_i}{\sigmastar}\right) \deltahat_i x_i\right\|_\infty \notag \\
& \le \left\|\left(I - \frac{\Thetahat}{\sigmahat \Ahat(\psi)} \left(\frac{1}{n} \sum_{i=1}^n \psi'\left(\frac{\epsilon_i}{\sigmastar}\right) x_i x_i^T \right)\right) (\betahat - \betastar)\right\|_\infty \notag \\
& \qquad + \left|\frac{1}{\sigmahat} - \frac{1}{\sigmastar}\right| \cdot \left\|\frac{1}{n} \sum_{i=1}^n \psi'\left(\frac{\epsilon_i}{\sigmastar}\right) \epsilon_i x_i\right\|_\infty \notag \\
& := A_1 + A_2.
\end{align}
We have
\begin{align*}
A_1 & \le \opnorm{\frac{\Thetastar}{\sigmastar A(\psi)} - \frac{\Thetahat}{\sigmahat \Ahat(\psi)}}_1 \cdot \left\|\sigmastar A(\psi) \Sigmastar - \frac{1}{n} \sum_{i=1}^n \psi'\left(\frac{\epsilon_i}{\sigmastar}\right) x_i x_i^T\right\|_{\max} \\
& \qquad \cdot \|\betahat - \betastar\|_1.
\end{align*}
By inequality~\eqref{EqnThetaRatio} and Lemma~\ref{LemSigma}, we know that the first factor is $\order\left(\frac{k \log p}{\sqrt{n}}\right)$. Furthermore, applying Lemma~\ref{LemMaxGauss}, and using Assumptions (A1)--(A3) and the fact that $\|\psi'\|_\infty < \infty$, the second factor is $\order\left(\sqrt{\frac{\log p}{n}}\right)$, w.h.p. (This is a close analog of Lemma~\ref{LemCovX}.) Hence, we conclude that
\begin{equation*}
A_1 = \order\left(\frac{k \log p}{\sqrt{n}} \cdot \sqrt{\frac{\log p}{n}} \cdot k \sqrt{\frac{\log p}{n}}\right),
\end{equation*}
which is $o_\mprob(1)$ under the assumed scaling. Next, we bound
\begin{equation*}
A_2 = \order\left(\frac{k \log p}{\sqrt{n}} \cdot \sqrt{\frac{\log p}{n}}\right) = o_\mprob(1),
\end{equation*}
using Lemma~\ref{LemConcX}. The desired result then follows.


\subsection{Proof of Theorem~\ref{ThmEfficient}}
\label{AppThmEfficient}

By Theorem~\ref{ThmAsympNorm}, the asymptotic variance of $\sqrt{n} (\bhat_\psi - \betastar)_J$ is equal to
\begin{equation*}
V_J = \frac{\E[\psi^2(\epsilon_i/\sigmastar)]}{\E[\psi'(\epsilon/\sigmastar)/\sigmastar]^2} \cdot (\Theta_x)_{JJ}.
\end{equation*}
We simply need to note that
\begin{equation*}
(\Theta_x)_{JJ} = \Big(\E\Big[\Big((x_i)_J - \E\left[(x_i)_J \mid (x_i)_{J^c}\right]\Big) \Big((x_i)_J - \E\left[(x_i)_J \mid (x_i)_{J^c}\right]\Big)\Big]\Big)^{-1},
\end{equation*}
so it suffices to prove the equivalence of the terms
\begin{align*}
V_1 & := \left(\E\left[\left(\frac{f_{\sigmastar}'(\epsilon_i)}{f_{\sigmastar}(\epsilon_i)}\right)^2\right]\right)^{-1}, \\
V_2 & := \frac{\E\left[\psi^2\left(\frac{\epsilon_i}{\sigmastar}\right)\right]}{\E\left[\frac{1}{\sigmastar}\psi'\left(\frac{\epsilon_i}{\sigmastar}\right)\right]^2},
\end{align*}
where $f_{\sigmastar}$ denotes the pdf of $\epsilon_i$. Taking $f$ to be the pdf of $\frac{\epsilon_i}{\sigmastar}$, we have
\begin{equation*}
f_{\sigmastar}(t) = \frac{1}{\sigmastar} f\left(\frac{t}{\sigmastar}\right), \quad \text{and} \quad f_{\sigmastar}'(t) = \frac{1}{(\sigmastar)^2} f'\left(\frac{t}{\sigmastar}\right),
\end{equation*}
so
\begin{equation*}
V_1 = \left(\E\left[\frac{1}{(\sigmastar)^2} \left(\frac{f'\left(\frac{\epsilon_i}{\sigmastar}\right)}{f\left(\frac{\epsilon_i}{\sigmastar}\right)}\right)^2\right]\right)^{-1}.
\end{equation*}
Furthermore, differentiating the equation $\psi(t) = \frac{f'(t)}{f(t)}$, we have
\begin{equation*}
\psi'(t) = \frac{f(t) f''(t) - (f'(t))^2}{(f(t))^2},
\end{equation*}
so
\begin{equation*}
V_2 = \frac{\E\left[\left(\frac{f'\left(\frac{\epsilon_i}{\sigmastar}\right)}{f\left(\frac{\epsilon_i}{\sigmastar}\right)}\right)^2\right]}{\left(\E\left[\frac{1}{\sigmastar} \cdot \frac{f\left(\frac{\epsilon_i}{\sigmastar}\right) f''\left(\frac{\epsilon_i}{\sigmastar}\right) - \left(f'\left(\frac{\epsilon_i}{\sigmastar}\right)\right)^2}{\left(f\left(\frac{\epsilon_i}{\sigmastar}\right)\right)^2}\right]\right)^2}.
\end{equation*}
Furthermore, the square root of the term in the denominator is equal to
\begin{align*}
\frac{1}{\sigmastar} \cdot \E\left[\frac{f(\epsilon) f''(\epsilon) - f'(\epsilon) f'(\epsilon)}{f(\epsilon)^2}\right] & = \int_{-\infty}^\infty f''(t) dt - \int_{-\infty}^\infty \frac{f'(t) f'(t)}{f(t)} dt \\
& = \left[f'(t)\right]_{-\infty}^\infty - \int_{-\infty}^\infty \frac{f'(t) f'(t)}{f(t)} dt \\
& = - \E\left[\left(\frac{f'\left(\frac{\epsilon_i}{\sigmastar}\right)}{f\left(\frac{\epsilon_i}{\sigmastar}\right)}\right)^2\right],
\end{align*}
from which we conclude that $V_1 = V_2$. Thus, we have $V_J = \widebar{V}$ as well, implying the desired property of asymptotic efficiency.

\subsection{Proof of Theorem~\ref{ThmConfReg}}
\label{AppThmConfReg}

The proof follows in a straightforward manner from Theorem~\ref{ThmAsympNorm}, which establishes the weak convergence statement
\begin{equation*}
\sqrt{n} P_J(\bhat_\psi - \betastar) \stackrel{d}{\longrightarrow} \frac{\sqrt{\E[\psi^2(\epsilon_i/\sigmastar)]}}{A(\psi)} \cdot \left((\Theta_x)_{JJ}\right)^{1/2} Z,
\end{equation*}
where $Z \sim N(0, I_m)$. Rearranging, we have
\begin{equation*}
\frac{\sqrt{n} A(\psi)}{\sqrt{\E[\psi^2(\epsilon_i/\sigmastar)]}} \cdot \left((\Theta_x)_{JJ}\right)^{-1/2} \cdot P_J(\bhat_\psi - \betastar) \stackrel{d}{\longrightarrow} Z.
\end{equation*}
We then use the following lemma, proved in Appendix~\ref{AppLemSlutsky}:
\begin{lem*}
\label{LemSlutsky}
Under the assumptions of the theorem, we have
\begin{multline*}
\frac{\Ahat(\psi)}{\sqrt{\frac{1}{n} \sum_{i=1}^n \psi^2\left((y_i - x_i^T\betahat)/\sigmahat\right)}} \cdot \frac{\sqrt{\E[\psi^2(\epsilon_i/\sigmastar)]}}{A(\psi)} \cdot \left(\Thetahat_{JJ}\right)^{-1/2} \cdot \left((\Theta_x)_{JJ}\right)^{1/2} \stackrel{\mprob}{\longrightarrow} 1.
\end{multline*}
\end{lem*}
Hence, by Slutsky's theorem, we also have
\begin{equation*}
\frac{\sqrt{n} \Ahat(\psi)}{\sqrt{\frac{1}{n} \sum_{i=1}^n \psi^2\left((y_i - x_i^T \betahat)/\sigmahat\right)}} \cdot \left(\Thetahat_{JJ}\right)^{-1/2} \cdot P_J(\bhat_\psi - \betastar) \stackrel{d}{\longrightarrow} Z.
\end{equation*}

Combined with equation~\eqref{EqnScaleBox}, we then have
\begin{multline*}
\lim_{n,p,k \rightarrow \infty} \mprob\Bigg(\frac{\sqrt{n} \Ahat(\psi)}{\sqrt{\frac{1}{n} \sum_{i=1}^n \psi^2\left((y_i - x_i^T \betahat)/\sigmahat\right)}} \cdot \left(\Thetahat_{JJ}\right)^{-1/2} \cdot P_J(\bhat_\psi - \betastar) \in \scriptB_{\alpha, J}\Bigg) = 1-\alpha.
\end{multline*}
Rearranging the argument inside the probability expression yields the desired result.


\section{Proofs of technical lemmas}
\label{AppEfficiency}

In this appendix, we provide the technical proofs of the remaining results used in the proofs of the main theorems.

\subsection{Proof of Lemma~\ref{LemSigma}}
\label{AppLemSigma}

We begin by writing
\begin{multline*}
|\sigmahat - \sigmastar| \le \left|\frac{1}{n} \sum_{i=1}^n (y_i - x_i^T \betahat)^2 - \frac{1}{n} \sum_{i=1}^n (y_i - x_i^T \betastar)^2\right| + \left|\frac{1}{n} \sum_{i=1}^n (y_i - x_i^T \betastar)^2 - \E[\epsilon_i^2]\right|.
\end{multline*}
We can bound the second term by $\order\left(\frac{\sigmastar}{\sqrt{n}}\right)$ via Chebyshev's inequality. Expanding and using the triangle inequality, we may bound the first term as
\begin{align*}
& \left|\frac{1}{n} \sum_{i=1}^n \left( \left(x_i^T (\betastar - \betahat) + \epsilon_i\right)^2 - \epsilon_i^2\right)\right|  \\
& \qquad \le \frac{1}{n} \left|\sum_{i=1}^n \left(x_i^T (\betahat - \betastar)\right)^2\right| + \frac{2}{n} \left|\sum_{i=1}^n \left(x_i^T (\betahat - \betastar)\right) \epsilon_i\right| \\
& \qquad \le (\betahat - \betastar)^T \Sigmahat (\betahat - \betastar) + 2 \left\|\frac{X^T\epsilon}{n}\right\|_\infty \|\betahat - \betastar\|_1.
\end{align*}
Using Lemma~\ref{LemConeDev} in Appendix~\ref{AppRSC} with $\Gamma = \Sigmahat - \Sigmastar$, we have
\begin{equation}
\label{EqnCovBound}
\left|(\betahat - \betastar)^T \Sigmahat (\betahat - \betastar)\right| \le C \|\betahat - \betastar\|_2^2,
\end{equation}
w.h.p., using Lemma~\ref{LemCovering} in Appendix~\ref{AppUseful} to verify the deviation condition. Furthermore, we have
\begin{equation*}
\left\|\frac{X^T \epsilon}{n}\right\|_\infty \le \|X\|_{\max} \cdot \frac{\|\epsilon\|_1}{n} = \order(\sqrt{\log p}) \cdot \left(\E[|\epsilon_i|] + \order\left(\frac{1}{\sqrt{n}}\right)\right),
\end{equation*}
using a Chebyshev argument and Lemma~\ref{LemXmax}. Altogether, we conclude that the expression is upper-bounded by
\begin{equation*}
\order\left(\frac{k \log p}{n}\right) + \order(\sqrt{\log p}) \cdot \order\left(k \sqrt{\frac{\log p}{n}}\right) = \order\left(\frac{k \log p}{\sqrt{n}}\right).
\end{equation*}
For the second statement, we simply write
\begin{align*}
\left|\frac{1}{\sigmahat} - \frac{1}{\sigmastar}\right| & = \frac{1}{\sigmahat \sigmastar} |\sigmahat - \sigmastar| \le \frac{1}{\sigmastar\left(\sigmastar - \order\left(\frac{k \log p}{\sqrt{n}}\right)\right)} \order\left(\frac{k \log p}{\sqrt{n}}\right) \\
& = \order\left(\frac{k \log p}{\sqrt{n}}\right).
\end{align*}


\subsection{Proof of Lemma~\ref{LemGLasso}}
\label{AppLemGLasso}

Note that we may choose the value of $c_2 > 0$ in Lemma~\ref{LemMaxGauss}, which then determines the constants $c$ and $C$, so for instance, the condition $\frac{\log(2p^2n))^7}{n} \le C_2 n^{-c_2}$ holds with $c_2 = \frac{1}{2}$ if $n \succsim \log^{15} p$.

For the remainder of the proof, we adapt an argument from Ravikumar et al.~\cite{RavEtal11}, suitable for the present setting.  The main technical argument is a primal-dual witness construction, which shows that the solution of the graphical Lasso restricted to the true support set also yields the unique global optimum when padded with zeros to obtain a $p \times p$ matrix. We only mention the necessary amendments to the arguments used in Ravikumar et al.~\cite{RavEtal11}; for more details, see the paper.

Following the proof of Theorem 1 in Ravikumar et al.~\cite{RavEtal11}, we denote $W = \Sigmahat - \Sigmastar$. By Lemma~\ref{LemCovX}, we have
\begin{equation*}
\|W\|_{\max} \le \frac{\alpha \lambda}{8},
\end{equation*}
w.h.p. Next, we define the matrix function
\begin{equation*}
R(\Delta) = \Thetahat^{-1} - \Theta^{*-1} + \Theta^{*-1} \Delta \Theta^{*-1}.
\end{equation*}
By Lemma 5 of Ravikumar et al.~\cite{RavEtal11}, we know that $\|\Delta\|_{\max} \le \frac{1}{3 \kappa_{\Sigmastar}d}$ implies that
\begin{equation*}
\|R(\Delta)\|_{\max} \le \frac{3k\|\Delta\|_{\max}^2 \kappa^3_{\Sigmastar}}{2}.
\end{equation*}
Lemma 6 of Ravikumar et al.~\cite{RavEtal11} then applies directly, as well, stating that if
\begin{equation}
\label{EqnCherry}
r := 2\kappa_{\Gammastar} (\|W\|_{\max} + \lambda) \le \min\left\{\frac{1}{3\kappa_{\Sigmastar} k}, \; \frac{1}{3 \kappa^3_{\Sigmastar} \kappa_{\Gammastar} k}\right\},
\end{equation}
we have
\begin{equation*}
\|\Thetahat - \Thetastar\|_{\max} \le r.
\end{equation*}
Note that the bound~\eqref{EqnCherry} holds by our assumption on the range of $\lambda$. In particular, we have
\begin{align*}
\|R(\Thetahat - \Thetastar)\|_{\max} & \le \frac{3\kappa^3_{\Sigmastar}k}{2} \|\Thetahat - \Thetastar\|_{\max}^2 \\
& \le \frac{3\kappa^3_{\Sigmastar}k}{2} \cdot 4 \kappa_{\Gammastar}^2 \left(\frac{\alpha}{8} + 1\right)^2 \lambda^2 \\
& \le \frac{\alpha \lambda}{8},
\end{align*}
by our assumptions. Lemma 4 of Ravikumar et al.~\cite{RavEtal11} then applies, implying the required strict dual feasibility result and the validity of the primal-dual witness construction argument.


\subsection{Proof of Lemma~\ref{LemApsi}}
\label{AppLemApsi}

Condition~\eqref{EqnConvDist} is easy to verify using the assumptions and the multivariate CLT (cf.\ Lemma~\ref{LemLindeberg}). Indeed, the mixed third moments of the summands are finite by assumption, and the variance of the limiting distribution is obtained via the calculation
\begin{align*}
\Var\left(P_J \cdot \frac{\Theta_x}{A(\psi)} \cdot \psi\left(\frac{\epsilon_i}{\sigmastar}\right) x_i\right) & = \frac{P_J \Theta_x}{A(\psi)} \cdot \E\left[\psi^2\left(\frac{\epsilon_i}{\sigmastar}\right)\right] \cdot \Sigmastar \frac{\Theta_x P_J^T}{A(\psi)} \\
& = \frac{\E[\psi^2(\epsilon_i/\sigmastar)]}{A^2(\psi)} \cdot P_J \Theta_x P_J^T.
\end{align*}

We now turn to proving condition~\eqref{EqnACond}. We first bound the estimation error $|\sigmahat - \sigmastar|$. We have the following result, proved in Appendix~\ref{AppLemSigma}:

\begin{lem*}
\label{LemSigma}
We have
\begin{equation*}
|\sigmahat - \sigmastar| = \order\left(\frac{k \log p}{\sqrt{n}}\right)
\end{equation*}
and
\begin{equation*}
\left|\frac{1}{\sigmahat} - \frac{1}{\sigmastar}\right| = \order\left(\frac{k \log p}{\sqrt{n}}\right).\end{equation*}
\end{lem*}

By the triangle inequality and a Taylor expansion, we have
\begin{align*}
|\Ahat(\psi) - A(\psi)| & \le \left|\frac{1}{\sigmahat} - \frac{1}{\sigmastar}\right| \left|\psi'\left(\frac{y_i - x_i^T \betahat}{\sigmahat}\right)\right| \\
& \qquad + \frac{1}{\sigmastar} \left|\frac{1}{n} \sum_{i=1}^n\left(\psi'\left(\frac{y_i - x_i^T \betahat}{\sigmahat}\right) - \psi'\left(\frac{y_i - x_i^T \betastar}{\sigmastar}\right)\right)\right| \\
& \qquad + \frac{1}{\sigmastar} \left|\frac{1}{n} \sum_{i=1}^n \psi'\left(\frac{\epsilon_i}{\sigmastar}\right) - \E\left[\psi'\left(\frac{\epsilon_i}{\sigmastar}\right)\right]\right| \\
& := A + B + C.
\end{align*}
We now bound each of the terms separately. Note that since $\|\psi'\|_\infty < \infty$ by assumption, term $C$ may be bounded directly by $\order\left(\frac{1}{\sqrt{n}}\right)$, w.h.p., using Chebyshev's inequality. Furthermore, we have
\begin{equation*}
A \le \|\psi'\|_\infty \cdot \left|\frac{1}{\sigmahat} - \frac{1}{\sigmastar}\right| = \order\left(\frac{k \log p}{\sqrt{n}}\right),
\end{equation*}
using Lemma~\ref{LemSigma}. Finally, defining
\begin{equation*}
\deltahat_i := \frac{y_i - x_i^T \betahat}{\sigmahat} - \frac{y_i - x_i^T \betastar}{\sigmastar} = \frac{x_i(\betastar - \betahat)}{\sigmahat} + \epsilon_i \left(\frac{1}{\sigmahat} - \frac{1}{\sigmastar}\right),
\end{equation*}
we may use a Taylor series expansion to write
\begin{align*}
B = \frac{1}{\sigmastar} \left|\frac{1}{n} \sum_{i=1}^n \psi''\left(\frac{\epsilon_i}{\sigmastar}\right) \deltahat_i + \frac{1}{n} \sum_{i=1}^n \frac{\psi^{(3)}(\uhat_i)}{2} \cdot \deltahat_i^2\right|,
\end{align*}
where $\uhat_i$ lies on the segment between $\frac{y_i - x_i^T \betastar}{\sigmastar}$ and $\frac{y_i - x_i^T \betahat}{\sigmahat}$, for each $i$. Applying the triangle inequality and H\"{o}lder's inequality then gives
\begin{align}
\label{EqnBBound}
\sigmastar B &
%
\le \frac{1}{\sigmahat} \left\|\frac{1}{n} \sum_{i=1}^n \psi''\left(\frac{\epsilon_i}{\sigmastar}\right) x_i\right\|_\infty \cdot \|\betahat - \betastar\|_1 + \left|\frac{1}{\sigmahat} - \frac{1}{\sigmastar}\right| \cdot \left|\frac{1}{n} \sum_{i=1}^n \psi''\left(\frac{\epsilon_i}{\sigmastar}\right) \epsilon_i\right| \notag \\
& \qquad + \|\psi^{(3)}\|_\infty \left(\frac{1}{n} \sum_{i=1}^n \left(\frac{x_i(\betahat - \betastar)}{\sigmahat}\right)^2 + \frac{1}{n} \sum_{i=1}^n \epsilon_i^2 \left(\frac{1}{\sigmahat} - \frac{1}{\sigmastar}\right)^2\right) \notag \\
& := B_1 + B_2 + B_3.
\end{align}
We claim that $B = \order\left(\frac{k \log p}{n}\right)$.

Using Lemma~\ref{LemConcX}, together with the estimation error bound~\eqref{EqnEll1}, we have
\begin{equation*}
B_1 = \order\left(\sqrt{\frac{\log p}{n}}\right) \cdot \order\left(k \sqrt{\frac{\log p}{n}}\right),
\end{equation*}
w.h.p. Furthermore,
\begin{equation*}
B_2 = \order\left(\frac{k \log p}{\sqrt{n}}\right) \cdot \order\left(\frac{1}{\sqrt{n}}\right),
\end{equation*}
using the error bound on $|\sigmahat - \sigmastar|$ and Chebyshev's inequality. Finally, we have
\begin{align*}
B_3 & \le \|\psi^{(3)}\|_\infty \left(\frac{1}{\sigmahat^2} (\betahat - \betastar)^T \Sigmahat (\betahat - \betastar) + \left|\frac{1}{\sigmahat} - \frac{1}{\sigmastar}\right|^2 \frac{1}{n} \sum_{i=1}^n \epsilon_i^2 \right) \\
& \qquad = \order\left(\left(\sqrt{\frac{k \log p}{n}}\right)^2 + \left(\frac{k \log p}{\sqrt{n}}\right)^2 \cdot \frac{1}{\sqrt{n}}\right) = \order\left(\frac{k \log p}{n}\right),
\end{align*}
using the same argument as in the proof of inequality~\eqref{EqnCovBound} in Lemma~\ref{LemSigma}.

Putting the results together, we have
\begin{align*}
|\Ahat(\psi) - A(\psi)| & = \order\left(\frac{1}{\sqrt{n}}\right) + \order\left(\frac{k \log p}{\sqrt{n}}\right) + \order\left(\frac{k \log p}{n}\right) \\
& = \order\left(\frac{k \log p}{\sqrt{n}}\right),
\end{align*}
as claimed.


\subsection{Proof of Lemma~\ref{LemSlutsky}}
\label{AppLemSlutsky}

Note that it suffices to show the following convergence results:
\begin{align}
\frac{\Ahat(\psi)}{A(\psi)} & \stackrel{\mprob}{\longrightarrow} 1, \label{EqnA} \\
\frac{\E[\psi^2(\epsilon_i/\sigmastar)]}{\frac{1}{n} \sum_{i=1}^n \psi^2\left((y_i - x_i^T \betahat)/\sigmahat\right)} & \stackrel{\mprob}{\longrightarrow} 1, \label{EqnB} \\
\left(\Thetahat_{JJ}\right)^{-1} (\Theta_x)_{JJ} & \stackrel{\mprob}{\longrightarrow} 1, \label{EqnC}
\end{align}
since we may combine the statements via Slutsky's theorem to obtain the desired result. Convergence results~\eqref{EqnA} and~\eqref{EqnC} are direct consequences of Lemmas~\ref{LemApsi} and~\ref{LemGLasso}, respectively, under the assumed sample size scaling. For convergence result~\eqref{EqnB}, we may use a parallel argument to the one employed to bound term $B$ in the proof of Lemma~\ref{LemApsi}. The only difference is that we use a Taylor expansion of $\psi^2$ rather than $\psi'$. Note that we have assumed $(\psi^2)''$ to be bounded. Since
\begin{align*}
(\psi^2)' = 2\psi \psi',
\end{align*}
the terms we need to control are of the form
\begin{equation*}
B_1' := \left\|\frac{1}{n} \sum_{i=1}^n \psi\left(\frac{\epsilon_i}{\sigmastar}\right) \psi'\left(\frac{\epsilon_i}{\sigmastar}\right) x_i\right\|_\infty, \quad B_2' := \left|\frac{1}{n} \sum_{i=1}^n \psi\left(\frac{\epsilon_i}{\sigmastar}\right) \psi'\left(\frac{\epsilon_i}{\sigmastar}\right) \frac{\epsilon_i}{\sigmastar} \right|,
\end{equation*}
the analog of terms $B_1$ and $B_2$ in inequality~\eqref{EqnBBound}. By Chebyshev's inequality and assumption (B2), we have $B_2' = \order\left(\frac{1}{\sqrt{n}}\right)$.


\section{Additional useful lemmas}
\label{AppUseful}

We begin with a lemma concerning the magnitude of the entries of the design matrix.

\begin{lem*}
\label{LemXmax}
Suppose assumption (A2) holds. Then
\begin{equation*}
\mprob\left(\|X\|_{\max} \ge c\sqrt{\log np}\right) \le \frac{c'}{np},
\end{equation*}
for constants $c, c' > 0$ depending on $C_1$ and $\max_j (\Sigma_x)_{jj}$.
\end{lem*}

\begin{proof}
By assumption (A2) and the triangle inequality, we have
\begin{align*}
\E\left[\exp\left(\frac{X_{ij}^2}{C_1}\right)\right] & \le \E\left[\exp\left(\frac{|X_{ij}^2 - (\Sigma_x)_{jj}|}{C_1}\right)\right] \E\left[\exp\left(\frac{(\Sigma_x)_{jj}}{C_1}\right)\right] \\
& \le 2 \E\left[\exp\left(\frac{\max_j (\Sigma_x)_{jj}}{C_1}\right)\right] := C_1',
\end{align*}
for all $1 \le j \le p$. Hence, for any $t > 0$, we may write
\begin{align*}
\mprob\left(X_{ij}^2 \ge t\right) & = \mprob\left(\exp\left(\frac{X_{ij}^2}{C_1}\right) \ge \exp\left(\frac{t}{C_1}\right)\right) \\
& \le \frac{\E[\exp(X_{ij}^2/C_1)]}{\exp(t/C_1)} \\
& \le C_1' \exp\left(-\frac{t}{C_1}\right),
\end{align*}
using Markov's inequality. Taking a union bound over $i$ and $j$ then gives
\begin{equation*}
\mprob\left(\|X\|_{\max}^2 \ge t\right) \le np \cdot C_1' \exp\left(-\frac{t}{C_1}\right),
\end{equation*}
and setting $t = 2C_1\log(np)$ yields the desired result.
\end{proof}

We now have a useful lemma concerning a Gaussian approximation of maxima.

\begin{lem*}
\label{LemMaxGauss}
[Gaussian approximation of maxima~\cite{CheEtal13}]
Suppose $\{\epsilon_i\}_{i=1}^n \subseteq \real^p$ are independent random vectors and one of the following conditions holds uniformly in $1 \le i \le n$ and $1 \le j \le p$:
\begin{itemize}
\item[(E3)] $\E[\epsilon_{ij}] = 0$, $\E[\epsilon_{ij}^2] \ge c_1$, and $\E\left[\exp\left(\frac{|\epsilon_{ij}|}{C_1}\right)\right] \le 2$
\item[(E4)] $\E[\epsilon_{ij}] = 0$, $\E[\epsilon_{ij}^2] \ge c_1$, and $\E\left[\max_{1 \le j \le p} \epsilon_{ij}^4\right] \le C_1$.
\end{itemize}
Also suppose $\frac{\left(\log(pn)\right)^7}{n} \le C_2 n^{-c_2}$. Then there exist positive constants $c$ and $C$, depending only on $c_1, c_2, C_1$, and $C_2$ such that
\begin{equation*}
\sup_{t \in \real} \left|\mprob\left(\max_{1 \le j \le p} E_j \le t\right) - \mprob\left(\max_{1 \le j \le p} Y_j \le t\right)\right| \le Cn^{-c},
\end{equation*}
where $E = \frac{1}{\sqrt{n}} \sum_{i=1}^n \epsilon_i$, and $Y \sim \scriptN(0, \E[\epsilon_i \epsilon_i^T])$ is a multivariate Gaussian vector with components $Y_j$.
\end{lem*}

Lemma~\ref{LemMaxGauss} leads to several useful concentration inequalities, which we collect here. In particular, using a union bound together with standard Gaussian tail bounds, we have
\begin{equation*}
\mprob\left(\max_{1 \le j \le p} Y_j \le t\right) \le 2p^2 \exp\left(-\frac{t^2}{2\sigma_Y^2}\right),
\end{equation*}
where $\sigma_Y = \max_{j,k} \E[\epsilon_{ij} \epsilon_{ik}]$, so we have
\begin{equation*}
\max_{1 \le j \le p} E_j \le 2\sigma_Y \sqrt{\log p},
\end{equation*}
with probability at least $1 - Cn^{-c} - \exp(-c'\log p)$.

We have the following result:
\begin{lem*}
\label{LemCovX}
Suppose assumptions (A1)--(A3) hold. Suppose $\frac{(\log (2p^2n))^7}{n} \le C_2n^{-c_2}$, for a constant $c_2 > 0$. Then
\begin{equation*}
\mprob\left(\left\|\frac{X^TX}{n} - \Sigma_x\right\|_{\max} \ge 2\sigma_{xx} \sqrt{\frac{\log p}{n}}\right) \le Cn^{-c} + \exp(-c'\log p),
\end{equation*}
for constants $C, c, c' > 0$ depending only on $c_1, c_2, C_1, C_2$, and $\sigma_{xx}^2$.
\end{lem*}

\begin{proof}
We apply Lemma~\ref{LemMaxGauss} with $\epsilon_i = \mvec\left(\pm \left(x_i x_i^T - \Sigma_x\right)\right)$, where for a matrix $A \in \real^{p \times p}$, the vector $\mvec\left(\pm A\right) \in \real^{2p^2}$ has entries $\left\{\pm A_{jk}\right\}_{1 \le j,k \le p}$. Note that the conditions (E3) (or (E4)) are satisfied by assumption (A2). Hence, if $\frac{\log^7 (2p^2n)}{n} \le C_2n^{-c_2}$, the lemma implies that
\begin{equation}
\label{EqnKS}
\sup_{t \in \real} \left|\mprob\left(\left\|\frac{X^TX}{n} - \Sigma_x\right\|_{\max} \ge \frac{t}{\sqrt{n}}\right) - \mprob\left(Y \ge t\right)\right| \le Cn^{-c},
\end{equation}
where $Y = \max_{1 \le j,k \le p} |Y_{jk}|$ and $Y_{jk} \sim N\left(0, \Var(x_{ij} x_{ik})\right)$. In particular, under assumption (A3) and using a union bound, we have
\begin{align*}
\mprob(Y \ge t) & \le \sum_{1 \le j,k \le p} \mprob(|Y_{jk}| \ge t) \le 2\sum_{1 \le j,k \le p} \exp\left(-\frac{2t^2}{\sigma_{xx}^2}\right) = 2p^2 \exp\left(-\frac{2t^2}{\sigma_{xx}^2}\right).
\end{align*}
Accordingly, we take $t = 2\sigma_{xx} \sqrt{\log p}$, to obtain
\begin{equation*}
\mprob(Y \ge t) \le \exp(-c'\log p),
\end{equation*}
Combining this bound with inequality~\eqref{EqnKS} implies the desired statement.
\end{proof}

We also have the following result:
\begin{lem*}
\label{LemCovering}
Suppose $\Gamma = \frac{X^TX}{n} - \Sigma_x$. Under assumptions (A1)--(A3) and the sample size scaling $\frac{(2k \log p)^7}{n} \precsim n^{-c}$, the deviation bound~\eqref{EqnDevAssump} holds with probability at least $1 - Cn^{-c} - \exp(-c'\log p)$.
\end{lem*}

\begin{proof}
We use a covering argument, similar to the argument employed at the end of Appendix~\ref{AppRSC}. The idea is simply to use an $\epsilon$-net argument by taking a union of $\epsilon$-nets over all $2k$-dimensional subspaces to obtain an analog of inequality~\eqref{EqnCovering}. The only appreciable difference is that instead of using Hoeffding's inequality to obtain simultaneous concentration of the $\left\{v_j^T \Gamma v_j\right\}$ terms, where $\{v_j\}$ is the $\epsilon$-net, we employ Lemma~\ref{LemMaxGauss}, where we define $\epsilon_i$ to be a $\order(p^{2k})$-dimensional vector with coordinates $\pm v_j^T \Gamma v_j$. We then have the desired high-probability bound under the assumed sample size scaling.
\end{proof}

Finally, we have a few more related bounds:
\begin{lem*}
\label{LemConcX}
Under assumptions (B1)--(B2), we have
\begin{align}
\left\|\frac{1}{n} \sum_{i=1}^n \psi\left(\frac{\epsilon_i}{\sigmastar}\right) x_i\right\|_\infty  = \order\left(\sqrt{\frac{\log p}{n}}\right), \label{EqnBanana}\\
\left\|\frac{1}{n} \sum_{i=1}^n \psi''\left(\frac{\epsilon_i}{\sigmastar}\right) x_i\right\|_\infty = \order\left(\sqrt{\frac{\log p}{n}}\right), \label{EqnOrange} \\
\left\|\frac{1}{n} \sum_{i=1}^n \psi\left(\frac{\epsilon_i}{\sigmastar}\right) \psi'\left(\frac{\epsilon_i}{\sigmastar}\right) x_i\right\|_\infty = \order\left(\sqrt{\frac{\log p}{n}}\right), \label{EqnPeach} \\
\left\|\frac{1}{n} \sum_{i=1}^n \psi'\left(\frac{\epsilon_i}{\sigmastar}\right) \frac{\epsilon_i}{\sigmastar} x_i\right\|_\infty = \order\left(\sqrt{\frac{\log p}{n}}\right), \label{EqnStrawberry}
\end{align}
with probability at least $1 - Cn^{-c} - \exp(-c'\log p)$.
\end{lem*}

\begin{proof}
All the inequalities are proved by applying Lemma~\ref{LemMaxGauss}, and letting $\epsilon_i$ be the vectorized version of the summands. For inequality~\eqref{EqnBanana}, note that
\begin{align*}
\E\left[\exp\left(\frac{\left|\psi\left(\frac{\epsilon_i}{\sigmastar}\right) x_{ij}\right|}{C_1}\right)\right] & \le \E\left[\exp\left(\frac{\psi^2\left(\frac{\epsilon_i}{\sigmastar}\right)}{2C_1} + \frac{x_{ij}^2}{2C_1}\right)\right] \\
& = \E\left[\exp\left(\frac{\psi^2\left(\frac{\epsilon_i}{\sigmastar}\right)}{2C_1}\right)\right] \E\left[\exp\left(\frac{x_{ij}^2}{2C_1}\right)\right] \\
& \le \sqrt{2} \cdot \sqrt{2},
\end{align*}
where the final inequality uses the assumptions and concavity of the square root in Jensen's inequality. Hence, we may apply Lemma~\ref{LemMaxGauss} to obtain the desired bound. Note that the constant prefactor in the upper bound will be equal to $\E\left[\psi^2\left(\frac{\epsilon_i}{\sigmastar}\right)\right] \cdot \max_{j,k} \E[x_{ij} x_{ik}]$.

The remaining inequalities are proved in a similar manner, so we omit the explicit arguments here.
\end{proof}

\begin{lem*}
\label{LemLindeberg}
[Multivariate Lindeberg-Feller CLT~\cite{Gre03}]
Suppose $\{x_n\}_{n \ge 1}$ are independent random vectors such that all mixed third moments are finite. Let $\E[x_i] = \mu_i$ and $\Var[x_i] = Q_i$, and define
\begin{equation*}
\widebar{\mu}_n = \frac{1}{n} \sum_{i=1}^n \mu_i, \qquad \text{and} \qquad \widebar{Q}_n = \frac{1}{n} \sum_{i=1}^n Q_i.
\end{equation*}
Suppose
\begin{equation*}
\lim_{n \rightarrow \infty} \widebar{Q}_n = Q \succeq 0,
\end{equation*}
and for every $i$,
\begin{equation*}
\lim_{n \rightarrow \infty} \left(\widebar{Q}_n\right)^{-1} \frac{Q_i}{n} = 0.
\end{equation*}
Then
\begin{equation*}
\sqrt{n} \left(\frac{1}{n} \sum_{i=1}^n x_i - \bar{\mu}_n\right) \stackrel{d}{\longrightarrow} \scriptN(0, Q).
\end{equation*}
\end{lem*}

\begin{lem*}
\label{LemMADsum}
Suppose $X$ and $Y$ are independent random variables, where $X$ has a symmetric, unimodal density. Then
\begin{equation*}
\MAD(X) \le \MAD(X + Y).
\end{equation*}
\end{lem*}

\begin{proof}
Without loss of generality, we assume the distribution of $X$ is symmetric around 0. Note that
it suffices to show that
\begin{equation*}
\mprob\Big(\left|X + Y - \med(X+Y)\right| \ge \MAD(X)\Big) \ge \frac{1}{2}.
\end{equation*}
Indeed, we will show this inequality holds for any fixed value $Y = y$:
\begin{equation}
\label{EqnMADineq}
\mprob\Big(\left|X + y - M\right| \ge \MAD(X)\Big) \ge \frac{1}{2},
\end{equation}
where we have denoted $M = \med(X+Y)$. We may then write the left-hand probability as
\begin{equation*}
\mprob\left(X \ge M-y + \MAD(X)\right) + \mprob\left(X \le M-y - \MAD(X)\right) := I + II.
\end{equation*}
Note that:
\begin{enumerate}
\item If $M-y \ge \MAD(X)$, we have
\begin{equation*}
II \ge \mprob(X \le 0) \ge \frac{1}{2}.
\end{equation*}
\item If $M - y \le -\MAD(X)$, we have
\begin{equation*}
I \ge \mprob(X \ge 0) \ge \frac{1}{2}.
\end{equation*}
\item Otherwise, suppose $0 \le M-y < \MAD(X)$ (the case when $M-y$ is negative is analogous). Consider
\begin{align*}
& \Big(I + II\Big) - \Big(\mprob(X \ge \MAD(X)) + \mprob(X \le -\MAD(X)\Big) \\
& \qquad = - \mprob\Big(\MAD(X) \le X < \MAD(X) + M-y \Big) \\
& \qquad \qquad + \mprob\Big(-\MAD(X) < X \le -\MAD(X) + M-y\Big) \\
& \qquad = - \mprob\Big(\MAD(X) \le X < \MAD(X) + M-y\Big) \\
& \qquad \qquad + \mprob\Big(\MAD(X) - M+y \le X < \MAD(X) \Big) \\
& \qquad \ge 0,
\end{align*}
where the final inequality comes from the assumption that the pdf of $X$ is unimodal, hence is a nonincreasing function on the interval $\MAD(X) \pm (M-y)$. We conclude that
\begin{equation*}
I+II \ge \mprob\Big(|X| \ge \MAD(X)\Big) \ge \frac{1}{2}
\end{equation*}
in this case, as well.
\end{enumerate}
This establishes inequality~\eqref{EqnMADineq}.
\end{proof}

\bibliography{refs.bib}

\end{document}